\documentclass[a4paper,10pt]{amsart}
\usepackage{amsmath,amsfonts,amssymb, amsthm, color}

\usepackage{epsfig}
\usepackage{verbatim}

\renewcommand{\deg}{\mathsf{deg}}

\DeclareMathOperator{\Aut}{\mathsf{Aut}}

\DeclareMathOperator{\ev}{\mathsf{ev}}

\newcommand{\Ext}{\operatorname{Ext}}
\newcommand{\Gal}{\operatorname{Gal}}
\newcommand{\End}{\operatorname{End}}
\newcommand{\Ind}{\operatorname{Ind}}
\newcommand{\Spec}{\operatorname{Spec}}
\newcommand{\GL}{\operatorname{GL}}
\newcommand{\SL}{\operatorname{SL}}
\newcommand{\Pic}{\operatorname{Pic}}
\newcommand{\Bun}{\operatorname{Bun}}
\newcommand{\Res}{\operatorname{Res}}
\newcommand{\K}{\operatorname{K}}
\newcommand{\Cone}{\operatorname{Cone}}

\def\Hom{\mathrm{Hom}}
\def\ZZ{\mathbb{Z}}
\def\PP{\mathbb{P}}
\def\TT{\mathbb{T}}
\def\Q{\mathbb{Q}}
\def\ff{\mathbb{F}}
\def\E{\boldsymbol{\mathcal{E}}}

\def\HH{\mathbf{H}}
\def\DH{\mathbf{DH}}
\def\U{\mathbf{U}}
\def\UU{\boldsymbol{\mathcal{E}}}
\def\tUU{\widetilde{\UU}}

\def\N{\mathbb{N}}

\def\q{\mathbf{q}}

\def\C{\mathbb{C}}
\def\AA{\mathbb{A}}
\def\OO{\mathbb{O}}
\def\sl{\mathfrak{sl}}

\def\Cc{\mathsf{C}}
\def\O{\mathcal{O}}
\def\F{\mathcal{F}}
\def\G{\mathcal{G}}
\def\H{\mathcal{H}}
\def\V{\mathcal{V}}

\def\x{\mathbf{x}}
\def\y{\mathbf{y}}
\def\z{\mathbf{z}}
\def\p{\mathbf{p}}
\def\o{\mathbf{o}}
\def\fqb{\overline{\mathbb{F}}_q}

\def\kk{\mathbf{k}}

\def\ll{\mathbf{l}}

\def\Z{\mathbf{Z}}
\def\s{\sigma}

\def\P{\mathcal{P}}
\def\tP{\widetilde{\mathcal{P}}}
\def\qqb{\overline{\mathbb{Q}}}
\def\ds{\displaystyle}

\def\Coh{\mathsf{Coh}}
\def\Frob{\mathrm{Frob}}
\def\frob{\mathsf{frob}}
\def\rk{\mathsf{rk}}
\def\gcd{\mathrm{g.c.d.}}
\def\lcm{\mathrm{l.c.m.}}
\def\AF{\mathsf{AF}}
\def\AFc{\mathsf{AF}^{\mathrm{cusp}}}

\def\Tor{\mathcal{T}\!or}
\def\Norm{\mathsf{Norm}}
\def\Fr{\mathrm{Fr}}
\def\Conv{\mathrm{\mathbf{Conv}}}

\newcommand{\He}{\mathcal{H}}

\newcommand{\wX}[1]{\widehat{\Pic^0(X_{#1})}} 
\newcommand{\wXF}[1]{\widetilde{\Pic^0(X_{#1})}} 

\def\f{\mathsf{f}}
\def\t{\mathsf{t}}
\def\ov{\overline}
\def\L{\mathcal{L}}

\def\D{\Delta}

\def\vvec{\mathrm{vec}}
\def\tor{\mathrm{tor}}
\def\ss{\mathrm{ss}}
\def\trho{{\tilde\rho}}
\def\tsigma{{\tilde\sigma}}

\def\txi{{\tilde\xi}}

\def \ot{\otimes}
\def\<{\langle}
\def\>{\rangle}

\newtheorem{thm}{\bf{Theorem}}[section]
\newtheorem{lemma}[thm]{Lemma}
\newtheorem{cor}[thm]{Corollary}
\newtheorem{prop}[thm]{Proposition}

\theoremstyle{remark}
\newtheorem{rem}[thm]{Remark}
\theoremstyle{definition}
\newtheorem{defin}[thm]{Definition}

\newcommand{\kA}{\mathcal{A}}

\newcommand{\kD}{\mathcal{D}}
\newcommand{\kE}{\mathcal{E}}
\newcommand{\kF}{\mathcal{F}}
\newcommand{\kG}{\mathcal{G}}
\newcommand{\kH}{\mathcal{H}}

\newcommand{\kL}{\mathcal{L}}
\newcommand{\kP}{\mathcal{P}}

\newcommand{\kM}{\mathcal{M}}

\newcommand{\kV}{\mathcal{V}}
\newcommand{\kW}{\mathcal{W}}

\numberwithin{equation}{section}

\input xy
\xyoption {all} 

\begin{document}

\title{Cusp Eigenforms and the Hall Algebra of an Elliptic Curve}
\author{Dragos Fratila}

\begin{abstract}
We give an explicit construction of the cusp eigenforms on an elliptic curve defined over a finite field using the theory of Hall algebras and the Langlands correspondence for function fields and $\GL_n$. As a consequence we obtain a description of the Hall algebra of an elliptic curve as an infinite tensor product of simpler algebras. We prove that all these algebras are specializations of a universal spherical Hall algebra (as defined and studied in~\cite{BS} and \cite{SV1}).
\end{abstract}

\maketitle

\tableofcontents

\section{Introduction}

\subsection{} The cusp eigenforms are the building blocks in the theory of automorphic forms. 
The celebrated Langlands correspondence for function fields and $\GL_n$, proved by Drinfel'd~\cite{Dr1} for $n=2$ and Laffforgue~\cite{Laff} for general $n$, puts in bijection the cusp eigenforms and the irreducible $l$-adic representations of the Galois group of the function field. Moreover, through this correspondence invariants from the automorphic side (the Hecke eigenvalues) match invariants from the Galois side (Frobenius eigenvalues). 
This allows one to deduce important information about a Galois representation by studying its automorphic counterpart and vice versa.

\vspace{0.1in}

This work is concerned with the unramified cusp eigenforms on an elliptic curve. The goal is to describe in detail the cusp eigenforms using the theory of Hall algebras. As a by-product we obtain that the Hall algebra of an elliptic curve is isomorphic (as a bialgebra) to an infinite tensor product of specializations of a universal spherical Hall algebra which was defined by Burban and Schiffmann in~\cite{BS}, as well as a  strengthening of the multiplicity 1 theorem conjectured by Kapranov in~\cite{Kap}. Moreover, as a corollary of the structure of the Hall algebra coupled with recent work of Schiffmann~\cite{S4} for the spherical Hall algebra, we can answer a question asked by Kapranov~[loc.cit] regarding higher functional equations satisfied by (unramified) Eisenstein series for $\GL_n$ over the function field of an elliptic curve.

\vspace{0.1in}

\subsection{} Let $X$ be a smooth, projective, geometrically irreducible curve defined over a finite field $\ff_q$. The study of the Hall algebra $\HH_X$ of coherent sheaves on $X$ was initiated by M. Kapranov in~\cite{Kap}. Using a classical observation of A. Weil we can express the space of vector bundles of rank $n$ on $X$ as the double coset space:
\[
\GL_n(K_X)\backslash \GL_n(\AA_X)/\GL_n(\OO_X)
\]
where $K_X, \AA_X,\OO_X$ are the function field, the ad\`eles respectively the integer ad\`eles of $X$. The elements of the vector bundle part $\HH_X^\vvec$ of the Hall algebra are non-ramified automorphic forms (with finite support) on the curve $X$. The multiplication is given by parabolic induction and the comultiplication by parabolic restriction. The classical notion of a cusp form in number theory translates into a very nice condition in terms of the coproduct of the Hall algebra. A fundamental problem in the theory of Hall algebras is to understand the (Hopf) algebra $\HH_X$ and to relate it to quantum groups. The importance of cusp forms for the Hall algebra business stems from the fact that they are a (minimal) set of generators for $\HH_X$.
\vspace{0.1in}

The classical Hecke operator associated to a point $x\in X$ on automorphic forms is, in the language of Hall algebras, given by multiplication with the characteristic function of the corresponding torsion sheaf. These Hecke operators naturally form a subalgebra inside $\HH_X$ which is the Hall algebra of the category of torsion sheaves on the curve. We will call this subalgebra the global Hecke algebra and we denote it by $\HH_X^\tor$. We have an isomorphism of algebras $\HH_X\simeq \HH_X^\vvec \ltimes \HH_X^\tor$.

\vspace{0.1in}

The space of cusp forms is stable under the Hecke operators. The elements of the eigenspaces of this action are called cusp eigenforms. The multiplicity 1 theorem (\cite{P-S, Sh}) says that the action of the Hecke algebra $\HH_X^\tor$ on the space of cusp forms is diagonalizable and that moreover every eigenspace is of dimension 1. In other words, for every character of $\HH_X^\tor$ there is at most one eigenvector (up to multiplication by scalars) corresponding to it. 
\vspace{0.1in}

In the seminal paper \cite{Kap} M. Kapranov considered the Eisenstein series associated to these cusp eigenforms and proved some quadratic functional equations satisfied by them as well as formulas for their coproduct. These functional equations in turn give relations between the coefficients of the Eisenstein series in the Hall algebra. 

\vspace{0.1in}

For the projective line $\PP^1_{\ff_q}$, due to a theorem of Grothendieck~\cite{Gr}, there are no cusp forms of rank $>1$. In this case Kapranov proved (see also~\cite{BK} for a more elementary approach) that the relations given by the functional equations suffice to give a complete description of the Hall algebra $\HH_{\PP^1_{\ff_q}}$: it is isomorphic to $U_v(\widehat{\mathfrak{n}})\ltimes \HH_{\PP^1_{\ff_q}}^\tor$ where $U_v(\widehat{\mathfrak{n}})$ is the nilpotent part (in Drinfel'd's new realization) of the quantum group $U_v(\widehat\sl_2)$ and $v=q^{-1/2}$.

\vspace{0.1in}

When $X$ is of genus $>0$ the functional equations are no longer enough to give a presentation of $\HH_X$. To simplify the problem we can define a new algebra $\U_X$, called the spherical Hall algebra, as being the subalgebra of $\HH_X$ generated by the coefficients of the simplest cusp eigenforms: the characteristic functions of the connected components of the Picard group, together with the characteristic functions of the set of torsion sheaves of (a fixed) degree $d$ for all $d\ge 1$. For instance, when $X=\PP^1_{\ff_q}$ we have that $\U_X$ is isomorphic to $U_v^+(\widehat{\sl_2})$ (in Drinfel'd's new realization).

The aim of this paper is to study the cusp eigenforms on an elliptic curve $X$ using some twisted versions of the spherical Hall algebra $\U_X$. The approach turns out to be fruitful for both theories: it enables us to write down explicitly all the (unramified) cusp eigenforms on $X$ and to give a presentation of the whole Hall algebra $\HH_X$.

\subsection{} For $X$ an elliptic curve there are several remarkable results that describe the spherical Hall algebra and connect it to  objects from different fields of mathematics. We shortly review them below for a motivational purpose.

\vspace{0.1in}

I. Burban and O. Schiffmann have proved in \cite{BS} that the Drinfel'd double $\DH_X$ of the Hall algebra  $\HH_X$ admits a natural $\SL_2(\ZZ)$ action and they have used this action to give a combinatorial presentation of $\U_X$ in terms of convex paths in the rank 2 lattice of integers. They have also proved that the spherical Hall algebra is a two parameter flat deformation of the algebra of diagonally symmetric functions $\C[x_1,x_2,\dots,y_1,y_2,\dots]^{\mathfrak{S}_\infty}$.

\vspace{0.1in}

In a more geometric direction, O. Schiffmann has constructed the algebra $\U_X$ (or rather a completion of it) as the (trace  \`a la Grothendieck's ''faisceaux-fonctions correspondence`` of the) Grothendieck group of a certain category of perverse sheaves on the moduli stack of coherent sheaves on $X$. This category of perverse sheaves is proved to be exactly the category generated by the simple factors of the principal Eisenstein sheaves as defined by Laumon in~\cite{Lau} for an elliptic curve.

\vspace{0.1in}

In the paper \cite{SV1} it was proved that the spherical Hall algebra is isomorphic to the stable limit of the spherical DAHA of $\GL_n$ when $n\to \infty$. Inspired by this result the authors of~[loc.cit.] exhibit a strong relationship between the symmetric Macdonald polynomials and the spherical Hall algebra and are able to give a geometric construction of the Macdonald polynomials as the traces of certain Eisenstein sheaves on the moduli space of semistable vector bundles on $X$. 

\vspace{0.1in}

Another interesting result is the connection between the spherical Hall algebra $\U_X$ and the equivariant $K$-theory of the Hilbert scheme of $\AA^2$. In \cite{SV2} Schiffmann and Vasserot proved that the convolution algebra in the equivariant $K$-theory of the Hilbert scheme of points of $\AA^2$ is isomorphic to the spherical Hall algebra and hence, by their previous result, to the stable limit of spherical DAHA for $\GL_n$. 

\vspace{0.1in}

We have already mentioned that if the genus of the curve is bigger than 0 then the quadratic functional equations for the Eisenstein series are no longer sufficient to give a presentation of the Hall algebra. In his paper~\cite{Kap} Kapranov asked what are the other functional equations satisfied by the Eisenstein series that enable us to give a presentation of the Hall algebra. The same problem can be stated for the spherical Hall algebra. In the case of an elliptic curve an answer to the problem was given by Schiffmann in~\cite{S4}. He proves that there is a cubic functional equation for the Eisenstein series and that this relation coupled with the quadratic functional equations provide a presentation of the spherical Hall algebra.

\subsection{} Let us describe the results of this paper in more detail. A word of warning is necessary: for the purpose of this introduction we will simplify some notations and we will be a bit lax about some technical details. References to precise statements and definitions in the text are made.

\vspace{0.1in}

From now on $X$ will denote an elliptic curve over a finite field $\ff_q$.

Let $\Coh(X)$ be the category of coherent sheaves on $X$. To every coherent sheaf $\kF$ on $X$ we associate its class in the numerical Grothendieck group $K_0'(X):=\ZZ^2$ given by $\ov\kF:=(\rk(\kF),\deg(\kF))$. 

For every $\nu\in\Q\cup\{\infty\}$ we denote by $\Cc_{\nu}$ the full subcategory of $\Coh(X)$ of semistable sheaves of slope $\nu$. It is an abelian category stable by extensions. From Atiyah's theorem (see~\cite{At}) the categories $\Cc_{\nu}$ are all equivalent and hence their Hall algebras are all isomorphic. 

For a class $\alpha=(r,d)\in\K_0'(X)$ let us define the following element:
\[
 \mathbf{1}_\alpha^{\ss}:=\sum_{\substack{\kF\in\Cc_{\nu}\\\ov\kF = \alpha}} [\kF]
\]
the characteristic function of the semistable sheaves of class $\alpha$, where $\nu:=d/r$.

The spherical Hall algebra $\U_X^+$ is defined to be the subalgebra of $\HH_X$ generated by the elements $\mathbf{1}_{(1,d)}^{\ss},d\in\ZZ$ and $\mathbf{1}_{(0,d)},d\in\ZZ_+$. 
The function $\sum_{d\in\ZZ} \mathbf{1}_{(1,d)}^\ss$ is the simplest example of a cusp eigenform. It can be proved that for any $\x\in\Z^+$ the elements $\mathbf{1}_{\x}^\ss$ are contained in $\U_X^+$.

We can define similar subalgebras for any cusp eigenform and then study them and see how they interact inside the whole Hall algebra. This is roughly the path taken in~\cite{Kap}. We will pursue a different direction inspired by~\cite{BS}. We will define subalgebras of $\HH_X$, called twisted spherical Hall algebras, by suitably twisting the generators of $\U_X^+$ with characters of the Picard groups. It will turn out that these subalgebras correspond exactly to the cusp eigenforms which we were looking for. Moreover, this approach gives us explicit formulas for all these cusp eigenforms.

\vspace{0.1in}

To this end, we define another set of generators $T_\x$ of $\U_X^+$ by the following generating series:
\[
 1+\sum_{r\ge 1}\mathbf{1}_{r\alpha}^\ss z^r = \exp\left(\sum_{r\ge1} \frac{T_{r\alpha}}{[r]}z^r\right)
\]
where $\alpha = (p,q)\in\Z^+$ is such that $\gcd(p,q)=1$.

It is clear that the elements $T_{\alpha},\alpha\in\Z^+$ generate the algebra $\U_X^+$. 

In \cite{BS} it was defined an action of the group $\SL_2(\ZZ)$ on the Drinfel'd double $\DH_X$ of the Hall algebra $\HH_X$. The newly defined generators $T_{\alpha}$ of $\U_X^+$ enjoy a symmetry property:
\[
 \gamma\cdot T_{\alpha} = T_{\gamma\cdot\alpha}
\]
where $\gamma\in\SL_2(\ZZ)$ and $\alpha\in\Z^+$ are such that $\gamma\cdot\alpha\in\Z^+$.

Since for torsion sheaves we have a well defined notion of support we can look at $T_{(0,d),x}$, the part of $T_{(0,d)}$ supported on the closed point $x\in X$. It turns out that $T_{(0,d)} = \sum_x T_{(0,d),x}$. Now using the above mentioned $\SL_2(\ZZ)$ action we can define similar elements $T_{\alpha,x}$ for any $\alpha\in\Z^+$ and any closed point $x\in X$. We still have the relation
\[
 T_{\alpha} = \sum_{x\in X}T_{\alpha,x}.
\]

It is not difficult to prove that the elements $T_{\alpha,x}$ for $\alpha\in \Z^+$ and $x\in X$ generate the \textit{whole} Hall algebra $\HH_X$.

Fix a positive integer $n\ge 1$. For a character of the Picard group $\Pic^0(X_n)$, where $X_n:=X\times\Spec(\ff_{q^n})$, we
define a twisted average of the elements $T_{\alpha,x}, x\in X$, where $\alpha\in n\Z^+$, by the formula:
\[
T_{\alpha}^{\trho} = \sum_{x\in X}\trho(x)T_{\alpha,x}.
\]
For precise definitions and notations see Section~\ref{S:twisted spherical}.

To a character $\rho\in\wX{n}$ we can associate the subalgebra $\U_X^{\trho,+}$ of $\HH_X$ generated by the elements $T_{\alpha}^\trho$ when $\alpha\in n\Z^+$. We call this algebra the twisted spherical Hall algebra (for a precise definition see Section~\ref{SS:def of twisted spherical}). If $n=1$ and $\rho$ is the trivial character of $\Pic^0(X)$ then we get the spherical Hall algebra as defined in \cite{BS}. 


We call a character $\rho\in\wX{n}$ primitive if it has a maximal orbit under the Galois group $\Gal(\ff_{q^n}/\ff_q)$. We denote the union over all $n\ge 1$ of the primitive characters by $\kP$.
The Langlands correspondence for elliptic curves tells us that, roughly, the cusp eigenforms of rank $n$ are parametrized by the primitive characters of $\Pic^0(X_n)$ (actually by their Galois orbits). One of our main results states the following 
\begin{thm}[see Theorem~\ref{T:cusp forms}]\,\qquad
\begin{enumerate}
 \item The automorphic form
\[
 T_n^\trho:=\sum_{d\in\ZZ} T_{(n,nd)}^\trho
\]
is a cusp eigenform for every primitive character $\rho\in\wX{n}$ 
\item This gives essentially\footnote{up to a $\C^\times$ action} all the cusp eigenforms on the elliptic curve $X$.
\end{enumerate}
\end{thm}
As an immediate consequence we obtain that the algebra $\HH_X$ is generated by the subalgebras $\U_X^{\trho,+}$ for $\rho\in\kP$.

Our second main result concerns the structure of $\U_X^{\trho,+}$ and of $\HH_X$. Along the lines of~\cite{BS} we prove:
\begin{thm}[see Theorem~\ref{T:str of Hall}]\quad
\begin{enumerate}
 \item The algebra $\U_X^{\trho,+}$, for $\rho\in\wX{n}$ a primitive character, has a presentation of the following form (see Subsection~\ref{ss:presentation twisted spherical} for a precise definition):

it is generated by $\{t_\x\mid \x\in\Z^+\}$ modulo the following relations:
\begin{enumerate}
	\item If $\x,\x'$ are proportional then \[
[t_\x,t_{\x'}]=0
\]
	\item If $\x,\y$ are such that $\delta(\x)=1$ and that the triangle $\Delta_{\x,\y}$ has no interior lattice point then
	\[
[t_\y,t_\x]=\epsilon_{\x,\y}c_{n\delta(\y)}\frac{\theta_{\x+\y}}{n(\nu^{-1}-\nu)}
\]
	where the elements $\theta_\z$ are defined by equating the coefficients of the following two series:	
	\[
\sum_{i\ge 0}\theta_{i\z_0}s^i=\exp\left(n(\nu^{-1}-\nu)\sum_{i\ge 1}t_{i\x_0}s^i\right)
\]
	for any $\x_0=(p,q)\in\Z^+$ such that $\gcd(p,q)=1$.
\end{enumerate}

\item The twisted spherical Hall algebras $\U_X^{\trho,+},\rho\in\kP,$ centralize each other and moreover we have an isomorphism of bialgebras
\[
\HH_X\simeq  \underset{\tilde\rho\in\P}{{\bigotimes}'} \U_X^{\trho,+}
\]
where the $\bigotimes'$ stands for restricted tensor product.
\end{enumerate}
\end{thm}
\vspace{0.1in}

\begin{rem}
The same description of the Hall algebra $\HH_X$ has been obtained recently in \cite{KSV} using shuffle algebras.
\end{rem}
\vspace{0.1in}

In Section 3.8 of~\cite{Kap}, Kapranov posed the following problem: find a presentation of the Hall algebra in terms of the homogeneous parts of the cusp eigenforms and the torsion sheaves appearing in the formula for the coproduct of these cusp forms. In our context and with the above notations the question is to find a presentation of the algebra $\U_X^{\trho,+}$ but only in terms of the generators $T_{(n,\pm nd)}^\trho, T_{(0,nd)}^\trho,d\in\ZZ_+$. Otherwise saying, it asks for the higher rank relations satisfied by the residues of the Eisenstein series associated to the cusp eigenform $T_n^\trho$. 

In \cite{S4} an answer to the above question for the spherical Hall algebra $\U_X$ and $X$ an elliptic curve is given. Our presentation of $\U_X^{\trho,+}$ and $\HH_X$ coupled with the results of~\cite{S4} give a complete answer to the above question for the whole Hall algebra in the case of elliptic curves.

\subsection{} Let us outline the contents of the paper. 
In Section 1 we introduce the notations and we recall Atiyah's theorem about vector bundles on elliptic curves. In Section 2 we define the Hall algebra $\HH_X$ of $X$ and its Drinfel'd double $\DH_X$. Then we recall the action of $\SL_2(\ZZ)$ on $\DH_X$. We recall de notion of a cuspidal element in the Hall algebra and the Hecke operators. These are analogous to the number theoretical situation. In Section 3 we define the twisted spherical Hall algebras which will be the main objects of study in this work. The purpose of Section 4 is to give a refresher on automorphic forms and Rankin-Selberg L-functions as well as to state some of the results of~\cite{Kap} relevant to our situation. We follow closely the presentation given in~\cite{Kap} to which we refer for full details. In Section 5 we state our main results. In Section 6 and Section 7, the technical heart of the paper, we work out the Langlands correspondence and the actions of the Hecke operators for elliptic curves as well as some formulas for the Hecke operators in the Hall algebra. In Section 8, using the formulas previously found in Section 6 and Section 7, we prove our main results. We included in the Appendix the proofs of some lemmas which didn't fit within the body of the article.

\section{Notations and Atiyah's theorem}
In this section we begin by fixing the notations we will use throughout. Then we recall Atiyah's classification of vector bundles on an elliptic curve and the theorems of Kuleshov, Geigle-Lenzing.
\subsection{}\label{ss:beginning} 
We fix once and for all an isomorphism of fields between $\C$ and $\qqb_l$, where $l\neq p$, which is the identity on $\overline{\Q}$.

We set $\Z:=\ZZ^2$, $\Z^*:=\Z-\{(0,0)\}$ and $\Z^\pm:=\pm\{(q,p)\in\Z: q>0\textrm{ or } q=0\textrm{ and }p\ge0\}$. For a point $\x=(q,p)\in\Z$ we put $\delta(\x)=\gcd(q,p)$. 

If $\lambda=(\lambda_1,\dots,\lambda_n)$ is a partition then we denote by $l(\lambda)=n$ its length and by $|\lambda|=\sum_i\lambda_i$ its size.

If $\nu\in\C^*-\{\pm1\}$ then we define the $\nu$-integers by the usual formula 
\[
[r]_\nu = \frac{\nu^r-\nu^{-r}}{\nu-\nu^{-1}}.
\]
If $\nu=v=q^{-1/2}$ then we denote $[r]_\nu$ simply by $[r]$.

For a finite abelian group $G$ we denote by $\widehat{G}$ its group of characters. 

\subsection{}
The letter $q$ will denote a power of a prime number $p$, $\kk=\ff_q$ the finite field with $q$ elements, $X$ an elliptic curve over $\ff_q$ (i.e. a smooth projective curve of genus 1 with a rational point) which is geometrically irreducible. We will fix the "origin" of the curve to be a rational point $x_0\in X(\ff_q)$. For a closed point $x\in X$ we will denote by $\O_{X,x}$ the ring of regular functions at $x$, by $\kk(x)$ its residue field, by $\O_x$ the torsion sheaf supported at $x$ whose stalk at $x$ is $\kk(x)$ and set $q_x=\#\kk(x)$.

We will denote by $\Coh(X)$ the category of coherent sheaves on $X$. For an extension of finite fields $\ff_{q^n}$ of $\ff_q$ we will denote by $X_n$ the fibered product $X\times_{\Spec(\ff_q)}\Spec(\ff_{q^n})$. By a sheaf we will always mean a coherent sheaf. Since we will only deal with (the group) $\Ext^1$ we will denote it simply by $\Ext$.

\vspace{0.1in}

As we fixed the origin we have (by Riemann-Roch) a bijective application $X(\ff_{q^n})=X_n(\ff_{q^n})\to \Pic^0(X_n)$ given by $x\mapsto \O_{X_n}(x-x_0)$. Therefore we can transport the group structure from $\Pic^0(X_n)$ to $X(\ff_{q^n})$ and $x_0$ will be the neutral element. Moreover, the inclusions $X(\ff_{q^m})\subseteq X(\ff_{q^n})$ for $m|n$ are compatible with the group structure. To avoid confusion with the addition of divisors, if $x,y$ are points in $X(\ff_{q^n})$, we will denote their sum in the group law by $x\oplus y\in X(\ff_{q^n})$.

\vspace{0.1in}
On the scheme $X$ we have the Frobenius endomorphism $\Fr_X$ which is the identity at the level of (topological) points and the raising at the $q$-th power at the level of functions. 

If $R$ is a $\ff_q$-algebra then this Frobenius acts on the $R$-points $X(R)$ of $X$ by composition: $\Fr_X\circ \ov{x}_R: \Spec{R} \stackrel{\ov{x}_R}{\longrightarrow} X\stackrel{\Fr_X}{\longrightarrow} X$. 
In particular for $R=\ff_{q^n}$ we get an action of the Frobenius on the set $X(\ff_{q^n})$ of  $\ff_{q^n}$-points of $X$ and moreover this action is compatible with the group structure of $X(\ff_{q^n})$. If $m|n$ are positive integers then we have an obvious identification $X(\ff_{q^m})=X(\ff_{q^n})^{\Fr_X^{m}}$. In particular, $X(\ff_{q^n})=X(\ov\ff_q)^{\Fr_X^n}$.

If we denote by $|X|$ the closed points of the scheme $X$ then we have an identification of sets $|X|\simeq X(\ov\ff_q)/\Fr_X$ where the quotient means that two $\ov\ff_q$-points of $X$ are identified if they have the same orbit under the Frobenius action $\Fr_X$. Similarly we have $|X_n|\simeq X(\ov\ff_q)/\Fr_X^n$. These two equalities allow us to define an action of $\Fr_X$ at the level of closed points of $X_n$. We denote this action by $\frob_n:|X_n|\to |X_n|$. It is clear that $|X_n|/\frob_n=|X|$.

\vspace{0.1in}

\subsection{}The Grothendieck group $\K_0(X)$ of $\Coh(X)$ is isomorphic to $\ZZ\oplus\Pic(X)$ and the isomorphism is given by $\kF\mapsto (\rk(\kF),\det(\kF))$. Moreover if we compose this morphism with the one sending a line bundle to its degree we get a group homomorphism $\K_0(X)\to\K_0'(X):=\ZZ^2$ given by $\kF\mapsto (\rk(\kF),\deg(\kF))$. We will call $\K_0'(X)$ the numerical Grothendieck group and for a sheaf $\kF$ we will denote by $\ov\kF$ its image in this group.

On $\K_0(X)$ we have the Euler bilinear form
\[\<\kF,\kG\>:=\dim\Hom(\kF,\kG)-\dim\Ext(\kF,\kG).\]

Since the canonical sheaf of an elliptic curve is trivial the Serre duality gives $\dim\Hom(\kF,\kG)=\dim\Ext(\kG,\kF)$. In particular we see that the Euler form is skew-symmetric.

The kernel of the map $\K_0(X)\mapsto \K'_0(X)$ is given by the radical of this bilinear form. Therefore the Euler form descends to a bilinear form on the numerical Grothendieck group which reads (by the Riemann-Roch formula):
\[\<(r_1,d_1),(r_2,d_2)\>=r_1d_2-r_2d_1.\]

\subsection{}\label{ss:frobenius on characters} 
Recall the Frobenius endo-morphism $\Fr_X$ of $X$. By extension of scalars we can define an endomorphism $\Fr_{X,n}:X_n\to X_n$ of $X_n$ which we will call the Frobenius of $X_n$ relative to $X$.

If $\Phi_n:\Pic^0(X_n)\to X(\ff_{q^n})$ is the isomorphism we fixed at the beginning (see Section \ref{ss:beginning}) then we have a commutative diagram:
\[\xymatrix{
\Pic^0(X_n)\ar[d]^{\Phi_n} & \Pic^0(X_n)\ar[l]_{\Fr_{X,n}^*}\ar[d]^{\Phi_n}\\
X(\ff_{q^n})\ar[r]_{\Fr_X} & X(\ff_{q^n})
}\]

The Frobenius $\Fr_{X,n}^*$ acts (by duality) on each group $\wX{n}$ and we will denote this action simply by $\Fr_{X,n}$. We will also denote by $\wXF{n}$ the quotient $\wX{n}/\Fr_{X,n}$.

If $n,m$ are positive integers such that $m|n$ then we define the relative norm maps $\Norm_m^n:\Pic(X_n)\to \Pic(X_m)$ by 
\[\Norm_m^n(\L):=\bigotimes_{i=0}^{n/m-1} (\Fr_{X,n}^*)^{mi}(\L).\]

The fact that the map is well defined follows from Galois descent: namely if we have a line bundle on $\ov X$ such that it is isomorphic to its Frobenius conjugate then the line bundle descends to a line bundle on $X$. For a proof see for example Lemma 1.2.1 in \cite{HN}.

By dualising we obtain relative norm maps between the character groups for which we will use the same notation 
\[\Norm_m^n:\wX{m}\to\wX{n}\]
hoping that it will not cause any confusion.

\vspace{0.1in}

Recall that we fixed a rational point $x_0$ on $X$ and therefore for any integers $n\ge 1$ and $d\in\ZZ$ we can identify canonically $\Pic^d(X_n)\equiv \Pic^0(X_n)$. This allows us to extend trivially any character $\rho_n$ of $\Pic^0(X_n)$ to a character of $\Pic(X_n)$ by putting $\rho_n(\O_{X_n}(x_0))=1$. Unless otherwise specified we will view the characters of $\Pic^0(X_n)$ as characters of $\Pic(X_n)$.

Set $\widetilde{X}:=\coprod_n \wXF{n}$. For a character $\tilde\rho\in\widetilde{X}$ we will say that it is a character of degree $n$ if $\trho\in\wXF{n}$.
Sometimes we will abuse the language and call the elements of $\widetilde{X}$ characters even if they are actually \emph{orbits} of characters.

\subsection{} For a sheaf $\F$ on $X$ we denote its slope by $\mu(\F)=\deg(\F)/\rk(\kF)$. We say that a sheaf $\F$ on $X$ is semistable (resp. stable) if for any subsheaf $0\subsetneq \G\subsetneq\F$ we have $\mu(\G)\le \mu(\F)$ (resp. $\mu(\G)<\mu(\F)$). For $\mu\in\Q\cup \{\infty\}$ we'll denote by $\Cc_\mu$ the full subcategory of $\Coh(X)$ of semistable sheaves of slope $\mu$. It is easy to see that $\Cc_\mu$ is an abelian category, closed under extensions and the simple objects are exactly the stable sheaves. Notice that $\Cc_\infty$ is the category of torsion sheaves. It follows from the definitions and from Serre duality that if $\mu>\nu$ then $\Hom(\Cc_\mu,\Cc_\nu)=0$ and $\Ext(\Cc_\nu,\Cc_\mu)=0$.

Atiyah gave a precise description of the category of sheaves on $X$, namely:

\begin{thm}[\cite{At}] Any sheaf on $X$ can be written essentially uniquely as a direct sum of semistable sheaves. For any $\mu\in\Q$ we have an exact equivalence of categories between $\Cc_\mu$ and $\Cc_\infty$. 
\end{thm}

The proof of Atiyah also provides an algorithm to compute these equivalences. His proof is for an algebraically closed field of characteristic zero but his methods can be extended to the finite field case. 

We would like to recall here a different (but of the same flavour) algorithm (or proof) given by Kuleshov~\cite{Kul} in the case of elliptic curves and by Lenzing-Meltzer~\cite{LM} in the case of weighted projective lines of genus 1 using mutations.

Let $\tau$ be a torsion sheaf of degree $n$ and let $\O$ denote the trivial line bundle on $X$. We will denote by $\V(\tau)$ the vector bundle which is the "universal extension" of $\tau$ by $\O$. This means that $V(\tau)$ fits in an exact sequence of the form:
\[0\to\O\ot\Ext(\tau,\O)^*\to \V(\tau)\to\tau\to 0\] and the class of $\V(\tau)$ in $\Ext(\tau,\O\ot\Ext(\tau,\O)^*) = \End(\Ext(\tau,\O))$ is the identity. 

For every $\mu\in \Q$ we have obvious equivalences of categories  $\Cc_\mu\to\Cc_{\mu+1}$ given by $\V\mapsto \V\ot_\O \O(x_0)$.

\begin{thm}[Kuleshov, Geigle-Lenzing]
The correspondence $\tau\mapsto \V(\tau)\ot \O(lx_0)$ from $\Cc_\infty$ to $\Cc_{l+1}$ gives an exact equivalence of categories for every $l\in\ZZ$. 
\end{thm}

In fact the above construction can be adapted to include all the slopes $\mu\in\Q$ but since we won't use it here we refer the interested reader to the original papers \cite{Kul}, \cite{LM}. See also Proposition 8.3 in \cite{S2} and Section 1 in \cite{S3}.

\section{Hall algebra of $X$}\label{def Hall alg} In this Section we will define the Hall algebra of $X$, its Drinfel'd double, the $\SL_2(\ZZ)$ action, the cuspidal elements and the Hecke operators. For a nice account of the theory of Hall algebras one can consult the paper \cite{S1}.

\subsection{} Fix a square root $v$ of $q^{-1}$. Let $\HH_X$ be the $\C$-vector space which has a basis given by $\{[\F]\}$ where $\F$ runs through the isomorphism classes of objects in $\Coh(X)$. 

\vspace{0.1in}

To a triple $(\F,\G,\H)$ of sheaves on $X$ we associate the number $P_{\F,\G}^{\H}$ of exact sequences 
\[0\to\G\to\H\to\F\to0.\]

For a sheaf $\kF$ we denote by $a_\kF$ the cardinal of its automorphism group.

On the vector space $\HH_X$ we define following \cite{Rin1} and \cite{Gr} an associative product

\[[\kF]\cdot[\kG]:=v^{-\<\kF,\kG\>}\sum_{\kH}\frac{P_{\kF,\kG}^{\kH}}{a_\kF a_\kG}[\kH],\]
a coassociative coproduct
\[\Delta([\kH]):=\sum_{\kF,\kG}v^{-\<\kF,\G\>}\frac{P_{\F,\G}^{\H}}{a_\H}[\F]\ot[\G]\]
a counit
\[\varepsilon{[\kF]}:=\delta_{\F,0}\]
and a hermitian form
\[([\F],[\G]):=\delta_{\F,\G}\frac1{a_\F}\]
which is in fact a non degenerate Hopf pairing, i.e. we have $(ab,c)=(a\ot b,\Delta(c))$.

The vector space $\HH_X$ endowed with these operations is a topological bialgebra. Topological here means that the coproduct doesn't actually take values in the tensor product $\HH_X\ot\HH_X$  but rather in some completion of it (see \cite{BS} Section 2.3 or \cite{S1} Lecture 1 for details).

\vspace{0.1in}
A more geometric way of defining the Hall algebra is as follows. Consider the space $\kF un_0(\kM_X,\C)$ of $\C$-valued functions with finite support on the set $\kM_X$ of isomorphism classes of objects of $\Coh(X)$. 
The set $\kM_X$ should be viewed as some "moduli" space for the objects of $\Coh(X)$. This space of functions identifies naturally as a vector space with $\HH_X$. We can endow $\kF un_0(\kM_X,\C)$ with a convolution product: namely let $f,g:\kM_X\to\C$ be two functions with finite support. Then we define:
\[
(f\star g)(\kF) = \sum_{\kG\subseteq\kF}v^{-\<\kF/\kG,\kG\>}f(\kF/\kG)g(\kG).
\]
It is easy to see that, using the identification of $\HH_X$ with $\kF un_0(\kM_X,\C)$, we obtain the same structure of algebra on $\HH_X$ as defined in the previous paragraph.

The coalgebra structure is defined (or proved to be, see \cite{S1} Proposition 1.5) in this new setting as:
\[
\Delta(f)(\kF,\kG) = v^{\<\kF,\kG\>-2\dim(\Ext(\kF,\kG))} \sum_{\xi\in\Ext^1(\kF,\kG)}f(\mathrm{cone}(\xi)[-1]).
\]

We will use these two definitions of the Hall algebra interchangeably depending on which one is more adapted to the given situation.

\vspace{0.1in}
The Hall algebra $\HH_X$ has a natural grading over the numerical Grothendieck group $\K_0'(X)=\ZZ^2$ given by $\HH_X[\alpha]:=\bigoplus_{\ov\F=\alpha}K[\F]$.

We will denote by $\HH_X^\tor:=\bigoplus_{d\ge 0}\HH_X[0,d]$ the sub algebra (in fact sub bialgebra) of torsion sheaves and by $\HH_X^\vvec:=\bigoplus_{(r,d)\in\ZZ^2,r>0}\HH_X^\vvec[r,d]$ the sub algebra of vector bundles. Observe though that $\HH_X^\vvec$ is not a sub coalgebra of $\HH_X$. We denote by $\pi^\vvec$ the projection map $\HH_X\to\HH_X^\vvec$.

\subsection{} Recall that when we have a bialgebra endowed with a Hopf pairing we can construct its Drinfel'd double. We will review here its definition in our context and we invite the interested reader to take a look at the books \cite{Jo} or \cite{KRT} for more details and proofs.

We use Sweedler's notation for the coproduct, namely if $a\in \HH_X$ we denote
\[\Delta(a)=a_{(1)}\ot a_{(2)}\]
the summation being understood.

We take two copies of $\HH_X$ which we will denote, in order to avoid confusion, by $\HH_X^+$ and $\HH_X^-$.

The Drinfel'd double of $\HH_X$  is defined as the quotient of the free product algebra $\HH_X^+ * \HH_X^-$ by the relations
\[b_{(1)}^-a_{(2)}^+(b_{(2)},\overline{a_{(1)}})=a_{(1)}^+b_{(2)}^-(b_{(1)},\overline{a_{(2)}})\]
where $(\,\,,\,\,)$ is the Green product and the overline means complex conjugation\footnote{the Drinfel'd double is defined for a bilinear Hopf pairing so we needed to make our sesquilinear Green form bilinear by conjugating the second term.} of the coefficients.

It can be proved that we have an isomorphism of vector spaces $\DH_X\simeq \HH_X^+\otimes \HH_X^-$.

\vspace{0.1in}

Seidel and Thomas ~\cite{Se-To} have constructed, using the Fourier-Mukai transform, an action of the braid group $B_3$ on the bounded  derived category $\mathcal{D}^b(\Coh(X))$. Moreover, in our case, this action is compatible with Atiyah's classification of coherent sheaves on $X$. This means that we can obtain all the above equivalences $\Cc_\infty\simeq \Cc_\mu$ as restrictions of certain Fourier-Mukai transforms. We will briefly review their construction and refer the reader to their paper~\cite{Se-To} for full details.

Let $\kE$ be a spherical object of the derived category $\kD^b(\Coh(X))$, i.e. $\kE$ satisfies $\Hom(\kE,\kE)=\Hom(\kE,\kE[1])=\kk$. Seidel and Thomas considered the functor $T_\kE:\kD^b(\Coh(X))\to\kD^b(\Coh(X))$ defined by
\[
 T_\kE(\kF) = \mathrm{cone}(\mathrm{ev}:\mathrm{RHom}(\kE,\kF)\otimes_\kk\kE \to \kF).
\]

The functor $T_\kE$ is an exact equivalence of categories for any spherical object $\kE$ (\cite{Se-To} Proposition 2.10). By \cite{Se-To}, Lemma 3.2 the functor $T_\kE$ is isomorphic to the Fourier-Mukai transform  with kernel 
\[
\mathrm{cone}(\kE^\vee\boxtimes\kE\to\O_\D)\in\kD^b(\Coh(X\times X)). 
\]
Observe that the objects $\O_X,\O_{x_0}$ are spherical and hence the functors $T_{\O_X},T_{\O_{x_0}}$ provide autoequivalences of the category $\kD^b(\Coh(X))$. These two equivalences satisfy a braid relation (see~\cite{Se-To}, Proposition 2.13):
\[
 T_{\O_{x_0}}T_{\O_X}T_{\O_{x_0}} = T_{\O_X}T_{\O_{x_0}}T_{\O_X}.
\]

\begin{prop}[\cite{Se-To}, \cite{BS} Proposition 1.2]
Let $\Phi:=T_{\O_{x_0}}T_{\O_X}T_{\O_{x_0}}$. Then $\Phi^2 = i^*[1]$ where $i$ is an involution of the curve $X$.
\end{prop}

By the above the group generated by $T_{\O_X},T_{\O_{x_0}},[1]$ in $\Aut(\kD^b(\Coh(X)))$ is isomorphic to the universal covering group $\widetilde\SL_2(\ZZ)$ of $\SL_2(\ZZ)$ which is given by the unique non trivial central extension of $\SL_2(\ZZ)$ by $\ZZ=\<[1]\>$.

In~\cite{BS} Section 3 Schiffmann and Burban proved the following important result:

\begin{thm}[\cite{BS} Corollary 3.10] The $\widetilde\SL_2(\ZZ)$ action on $\kD^b(\Coh(X))$ descends to an action of $\SL_2(\ZZ)$ by algebra automorphisms on $\DH_X$.
\end{thm}
In the sequel we will exploit extensively this symmetry of the Hall algebra.

\vspace{0.1in}

\begin{rem}Recently, T. Cramer~\cite{Cr} has extended this result by proving more generally that any derived auto-equivalence of a hereditary category induces an automorphism of the Drinfel'd double of the Hall algebra.
\end{rem}

\subsection{} In this paragraph we define the notion of cuspidality for elements in the Hall algebra. This notion is equivalent to the usual definition of cusp automorphic forms in number theory. We will prove some basic results about these elements.

\begin{defin} An element $f\in\HH_X[r,d]^{\vvec}, r>0$, is cuspidal if for any vector bundles $\kV,\kW$ on $X$ we have $(f,[\kV]\cdot[\kW])=0$.
\end{defin}

Observe that $[\kL]$ is cuspidal for any line bundle $\kL$.

The above condition is equivalent to $\Delta(f)\in \HH_X^{\vvec}\ot\HH_X^{\tor}+\HH_X^{\tor}\ot\HH_X^{\vvec}$.

\begin{prop}\label{P:cusps generate H_X} The algebra $\HH_X$ is generated by $\HH_X^\tor$ together with the set of cuspidal elements.
\end{prop}

\begin{proof} Denote by $B$ the subalgebra generated by the cuspidals and by the torsion sheaves and by $B^\perp$ its orthogonal in $\HH_X$ with respect to Green's scalar product. Since this scalar product is non-degenerate it suffices to prove that $B^\perp=0$.

Let $f\in B^\perp, f\neq 0$ be of smallest rank. As $f$ is not cuspidal there exist vector bundles $\kV,\kW$ of smaller rank such that $(f,\kV\cdot\kW)\neq 0$. But since $f$ was chosen to be of minimal rank in $B^\perp$ it follows that $\kV$ and $\kW$ are in $B$ and therefore $f\not\in B^\perp$. Contradiction.
\end{proof}

\vspace{0.1in} For each integer $r\ge 1$ denote by $\HH_X^{\le r}$ (resp. $\HH_X^{<r}$) the subalgebra of $\HH_X$ generated by the torsion sheaves plus all the vector bundles of degree $\le r$ (resp. $<r$). 

\begin{prop}\label{P:cusp perp} An element $f\in\HH_X[r,d]$ is a cusp form if and only if $f\in(\HH_X^{< r})^{\perp}$, where $\perp$ means the orthogonal with respect to Green's form.
\end{prop}
\begin{proof}
If $f$ is a cusp form of rank $r$ then, by the definition, we have $(f,\kV\cdot\kW)=0$ for any $\kV,\kW$ vector bundles of rank $<r$. By linearity it follows that $f\perp\HH_X^{<r}$.

Conversely, let $f\in\HH_X[r,d]\cap(\HH_X^{<r})^{\perp}$. Then $f$ is orthogonal to any element in $\HH_X^{<r}$, in particular it is orthogonal to any product $\kV\cdot\kW$ where $\kV,\kW$ are vector bundles of rank $<r$. Therefore $f$ is a cusp form.
\end{proof}

\begin{lemma}\label{L:Delta(cusp)=cusp}
Let $f\in \HH_X[r,d]$ be a cusp form. Write \[
\Delta(f) = f\ot 1+\sum_{i,l} \theta_{i,l}\ot f_{i,l}
\]
where $\theta_{i,l}\in \HH_X[0,l]$ and $f_{i,l}\in\HH_X[r,d-l]$ such that $\theta_{i,l}$ are mutually orthogonal (w.r.t. the Green form). Then the $f_{i,l}$ are cusp forms.
\end{lemma}
\begin{proof}
We will proceed by contradiction. Suppose there existed some $i_0,d_0$ such that $f_{i_0,l_0}$ would not be cuspidal. Then there would exist $a,b\in \HH_X^{>0}$ such that $(f_{i_0,l_0},a\cdot b)\neq 0$ 

Therefore we have that $(\Delta(f),\theta_{i_0,l_0}\ot a\cdot b) = \|\theta_{i_0,l_0}\|^2 (f_{i_0,l_0},a \cdot b)\neq 0$. But by using the Hopf property of Green's form we get that $(\Delta(f),\theta_{i_0,l_0}\cdot a\ot b)\neq 0$ which is a contradiction with the cuspidality of $f$.

\end{proof}

\subsection{Hecke operators}
In this Subsection we introduce the action of the Hecke operators and we state some of their properties.

\subsection{} Recall that $\HH_X^\tor$ is the Hall algebra of torsion sheaves on $X$ and that $\HH_X^\vvec$ is the subalgebra of $\HH_X$ that consists of functions supported on the vector bundles.

\begin{defin} The Hecke operators are given by the application $\H:\HH_X^\tor\otimes\HH_X^\vvec \to \HH_X^\vvec$ which is defined by:
\[
\H_\tau(f)(\V'):=\sum_{\substack{\V\subseteq\V'\\\V'/\V\simeq\tau}}f(\V)
\]
where $\tau\in\HH_X^\tor$ and $f\in\HH_X^\vvec$.
\end{defin}
We call $\H_\tau$ the Hecke operator associated to $\tau\in\HH_X^\tor$ and $\HH_X^\tor$ the algebra of Hecke operators.

As $\HH_X$ is an associative algebra we see that the Hecke operators make $\HH_X^\vvec$ a left $\HH_X^\tor$-module. It is the structure of this module (or rather of a completion of it) that is very important in the classical Langlands correspondence. The determination of its structure as a module over the algebra of Hecke operators is equivalent to the understanding of the automorphic side in the Langlands correspondence.

\vspace{0.1in}

\begin{prop}\label{P:Hecke action and mult with torsion} If $\tau\in\HH_X[0,d]\subset \HH_X^\tor$ and if $f\in\HH_X[n,d']^\vvec$ then
\[\H_\tau(f)=v^{-rd}\pi^\vvec(\tau\cdot f)\in\HH_X[r,d+d']^\vvec\]
where the multiplication is made in the Hall algebra $\HH_X$.
\end{prop}
\begin{proof}
This follows immediately by the definition of the product in the Hall algebra and by the above formula for the Hecke operators.
\end{proof}

Recall that a primitive element in a coalgebra $(C,\Delta,\epsilon)$ is an element $x\in C$ that satisfies $\Delta(x)=x\ot 1+1\ot x$.

For the primitive elements of $\HH_X^\tor$ we actually have a nicer description of the associated Hecke operator:

\begin{lemma}\label{L:Hecke op primitive torsion} Let $\tau\in\HH_X[0,d]$ be a primitive element and let $f\in \HH_X^\vvec[r,d']$. Then we have
\[\H_\tau(f) = v^{-rd}[\tau,f]\]
where the commutator is taken in the Hall algebra $\HH_X$.
\end{lemma}
\begin{proof}
By the previous proposition all we need to prove is that $[\tau,f]=\pi^\vvec(\tau\cdot f)$, or in other words that the element $[\tau,f]$ of the Hall algebra is supported on the vector bundles. It is enough to do this for $f$ of the form $f=\V$ where $\V$ is a vector bundle of rank $n$ and degree $d'$.

Let $\F$ be a coherent sheaf of rank $n$ which is not a vector bundle. We want to prove that $\<[\tau,\V],\F\>=0$. We can write $\F = v^{d}\V'\cdot \tau'$ for some vector bundle $\V'$ a (nontrivial) torsion sheaf $\tau'$ and an integer $d\in\ZZ$.

Write $\Delta(\V):=\V\ot 1+\sum c_{\F',\kW'}\F'\otimes\kW'$ where $\kW'$ are nonzero vector bundles and $\F'$ are coherent sheaves.

By the Hopf property of the Green pairing we have
\begin{eqnarray*}
\<[\tau,\V],\F\> & = & v^{d}\<[\tau\ot 1+1\ot\tau,\Delta(\V)],\V'\ot\tau'\>\\
& = & v^{d}\<[\tau\ot 1+1\ot\tau,\V\ot 1],\V'\ot\tau'\>+\\
& & +v^{d}\sum c_{\F',\kW'}\<[\tau\ot 1+1\ot \tau,\F'\ot\kW'],\V'\ot\tau'\>\\
& = & v^{d}\<[1\ot\tau,\V\ot 1],\V'\ot\tau'\>\\
& = & 0
\end{eqnarray*}
where we used the fact that $\tau'$ is a nonzero torsion sheaf and that $\kW'$ are nonzero vector bundles.
\end{proof}

\begin{cor} Let $\tau_1,...,\tau_s\in\HH_X^\tor$ be primitive elements of degree $d_1,...,d_s$ and let $\V$ be a vector bundle of rank $r$. Then 
\[
\H_{\tau_1\cdot...\cdot\tau_s}(\V)=v^{-r(d_1+...+d_s)}[\tau_1,[\tau_2,....,[\tau_s,\V]...]]. 
\]

\end{cor}

\begin{proof}
From the associativity of the product in the Hall algebra and the previous Lemma we have $\He_{\tau_1\cdot\tau_2} = \He_{\tau_1}\circ\He_{\tau_2}$.
\end{proof}

\section{Twisted spherical Hall algebras}\label{S:twisted spherical}
We define in this Section a generalization of the spherical Hall algebra as introduced and studied in \cite{BS}. We will call them twisted spherical Hall algebras and they will be one of the main objects of study in this work.

\subsection{} Let us begin by recalling some properties of the classical Hall algebra (i.e. the Hall algebra of finite modules over a discrete valuation ring).

If $\ll$ is a finite field we denote by $u$ a square root of $\#\ll^{-1}$. Denote by $\kA_\ll$ the category of finite modules over the discrete valuation ring $A:=\ll[[t]]$. There exists a unique (up to isomorphism) indecomposable module of length $r$, denoted $I_{(r)}$, which is defined as the quotient $A/t^rA$. For a partition $\lambda=(\lambda_1,\dots,\lambda_n)$ we denote by $I_\lambda:=I_{(\lambda_1)}\oplus\dots\oplus I_{(\lambda_n)}$. The collection $\{I_{\lambda}\}_\lambda$ where $\lambda$ runs over all partitions is a complete collection of representatives for the isomorphism classes of objects of $\kA_\ll$. The Hall algebra of the category $\kA_\ll$ is discussed at full length in the book \cite{Mac}, Chap. II and Chap. III. Let us denote by $\Lambda_t$ the Macdonald's ring of symmetric functions over $\C[t^{\pm1}]$ and by $e_\lambda$ (resp. $p_\lambda$) the elementary (resp. power-sum) symmetric functions.

We summarize the properties we need about $\HH_{\kA_\ll}$ in the following proposition:

\begin{prop}[\cite{Mac}]\label{P:classical Hall & Macd} The assignment $I_{(1^r)}\mapsto u^{r(r-1)}e_r$ extends to a bialgebra isomorphism $\Psi_\ll:\HH_{\kA_\ll}\to\Lambda_t|_{t=u^2}$. Set $F_r:=\Psi_\ll^{-1}(p_r)$. Then 
	\begin{enumerate}
		\item $F_r = \ds\sum_{|\lambda|=r}n_u(l(\lambda)-1)I_{\lambda}\quad\text{where}\quad n_u(l) = \prod_{i=1}^l(1-u^{-2i})$
		\item $\Delta(F_r) = F_r\ot 1+1\ot F_r$, and $\{F_r, r\in\N\}$ constitutes a basis of primitives in the coalgebra $\HH_{\kA_\ll}$.		
		\item $\ds (F_r,F_s) = \delta_{r,s}\frac{r u^r}{u^{-r}-u^r}$.
	\end{enumerate}
\end{prop}
\begin{proof} The proofs may be found in \cite{Mac} III.7 Ex. 2, I.5 Ex. 25 and III.4(4.11).
\end{proof}

\subsection{}\label{S:def of T_r,d}
Let $x$ be a closed point of $X$. Denote by $\kk_x$ its reside field. We will denote by $|x|$ the degree of $x$, which by definition is the degree of the field extension $\kk_x/\kk$. 
Consider the category $\Tor_x$ of torsion sheaves on $X$ supported at $x$. We have an equivalence of categories $\Tor_x\simeq \kA_{\kk_x}$ which provides us with an isomorphism of bialgebras 
$\Psi_{\kk_x}:\HH_{\Tor_x}\to\Lambda_t|_{t=v^{2|x|}}$. 

For $r\in\N$ and $x\in X$ we define the elements $T_{(0,r),x}\in\HH_X$ by the formula
\[T_{(0,r),x}:=\left\{\begin{array}{ll}0&
\text{if}\quad |x|\not| \,r\\
\ds\frac{[r]}{r}|x|\Psi^{-1}_{\kk_x}(p_{\frac{r}{|x|}})& \text{if}\quad |x|\mid r
\end{array}\right..
\]
Explicitly, using Proposition \ref{P:classical Hall & Macd} (1), we have
\[
T_{(0,r),x}:=\frac{[r]|x|}{r}\sum_{|\lambda| = r/|x|}n_{u_x}(l(\lambda)-1)\O_x^{(\lambda)}.
\]

Recall that we have an action of $\SL_2(\ZZ)$ on the algebra $\DH_X$. For any $\x\in\ZZ^2-\{0\}$ we define the elements $T_{\x,x}$ by translating the above $T_{(0,r),x}$'s using this action. 
More precisely, for $\gamma\in\SL_2(\ZZ)$ we define $T_{\gamma\cdot(0,r),x} := \gamma\cdot T_{(0,r),x}$.

\vspace{0.1in}

\begin{defin}\label{D:Trho}
For a character $\trho\in\wXF{n}$ and a closed point $x\in X$ we define \[\trho(x):=\frac1n\sum_{i=0}^{n-1}\rho((\Fr_{X,n}^*)^i(\O_{X_n}(x')))=\frac1{d}\sum_{i=0}^{d-1}\rho((\Fr_{X,n}^*)^i(\O_{X_n}(x')))\]
where $x'\in X_n$ is a closed point that sits above $x$ and $d=\gcd(|x|,n)$. 
\end{defin}
This definition does not depend on the representative of $\trho$ nor on the point $x'$ we chose since it is an averaging over all these possible choices.

\begin{defin}
For a character $\trho\in\widetilde X$ (see Paragraph~\ref{ss:frobenius on characters}) and for a point $\x\in\Z^*$ we define the element
\[
T_{\x}^{\trho}:=\sum_{x\in X} \trho(x)T_{\x,x}.
\]
\end{defin}

\begin{prop} The set $\{T_{\x}^{\trho} \mid \x\in\Z^+, \trho \in \wXF{d}, d = \delta(\x)\}$, generates the Hall algebra $\HH_X$. Similarly, the set $\{T_{\x}^{\trho} \mid \x\in\Z^*,\, \trho \in \wXF{d},\, d = \delta(\x)\}$ generates the double Hall algebra $\DH_X$.
\end{prop}

\begin{proof} We only prove the second part, the first one being a consequence. Let's denote for the moment by $B$ the subalgebra defined above.

It's clear that $T_{(0,d)}^{\rho_d}, d>0,\rho_d\in\wXF{d}$ generate the Hall algebra of the torsion sheaves because of the fact that, in general, any function on a finite commutative group can be written as a linear combination of characters.

Now using the action of the $\SL_2(\ZZ)$ on $\DH_X$ and its compatibility with the definition of $T_\x^{\tilde\rho}$ we see that any semistable sheaf is in the algebra $B$ and since the Hall algebra is generated by the semistable sheaves we see that $\DH_X\subseteq B$.
\end{proof}

\subsection{} In this paragraph we take a small detour to give a useful description of the primitive elements of the (bi)algebra of Hecke operators. We will use this description in the course of the proof of our main results.

\begin{lemma} A basis for the primitive elements in the Hall algebra is given by $T_{(0,d),x},x\in X,d\ge 1$. Moreover, these elements generate the Hall algebra of torsion sheaves.
\end{lemma}
\begin{proof} First of all it is clear that a primitive element should be supported on a single point. Then the statement follows from Proposition \ref{P:classical Hall & Macd} and the fact that the power-sum functions are a basis for the primitive elements of the Macdonald's symmetric functions algebra and that they generate it as an algebra. To conclude we only need to observe that we can write $\HH_X^\tor$ as a (commutative) restricted tensor product $\HH_X^\tor\simeq \bigotimes'_{x\in X}\HH_{X,x}^\tor$ where $\HH_{X,x}^\tor$ is the Hall algebra of torsion sheaves supported at $x$.
\end{proof}

\begin{cor}\label{C:basis primitive torsions} A basis for the primitive elements in the Hall algebra is also given by $\{T_{0,d}^{\tilde\rho},d\ge 1,\tilde\rho\in\wXF{d}\}$
\end{cor}
\begin{proof} This statement is clear because the characters of an abelian group are a basis for the space of functions on this group.
\end{proof}

\subsection{} We introduce the notion of primitivity for characters and state some of their basic properties.

\begin{defin} We call a character $\rho_n\in\wX{n}$ primitive (of degree $n$) if its orbit under the Frobenius $\Fr_{X,n}$ is of maximal cardinal, i.e. if it is of cardinal $n$. 
\end{defin}

We'll denote the set of primitive characters of degree $n$ modulo the action of Frobenius by $\kP_n$ and the set of all primitive characters modulo the Frobenius by $\P:=\coprod_{n\ge 1} \P_n$.

Observe that all the characters of degree 1, i.e. those which are in $\wX{}=\wXF{}$, are primitive.
The following lemma is proved in the Appendix (see Lemma~\ref{L:proof prim char}).

\begin{lemma}\label{L:prim char} A character $\rho\in\wX{n}$ is primitive if and only if there does not exist a character $\chi\in\wX{d},\, d<n,\, d|n$,  such that $\rho=\Norm_d^n(\chi)$.
\end{lemma}

\begin{cor}\label{C:prim char} If $\rho\in\wX{n}$ then either it is primitive or there exists a primitive character $\chi\in\wX{d},d|n,d<n$, such that $\rho=\Norm_d^n(\chi)$.
\end{cor}
\begin{proof}
We apply the previous Lemma and we take the smallest possible $d$ which satisfies the property that there exists $\chi\in\wX{d}$ such that $\rho = \Norm_d^n(\chi)$. Then if $\chi$ is not primitive, by the same Lemma, there exists a character $\chi'\in\wX{d'}$ where $d'<d$ such that $\chi = \Norm_{d'}^d(\chi')$. This contradicts the minimality of $d$.
\end{proof}

\subsection{}\label{SS:def of twisted spherical} We now have all the ingredients to define the twisted spherical Hall algebras.
\begin{defin} Let $n\ge 1$ and $\trho\in\kP_n$ be a primitive character. We define the algebra $\U_{X}^{\trho}$, called the twisted spherical Hall algebra of $X$ and character $\trho$, as being the subalgebra of $\DH_X$ generated by the elements:
\[\left\{T_{n\x}^{\Norm_{n}^{n\delta(\x)}(\trho)},\, \x\in\Z^*\right\}.\]

We define similarly the positive (resp. negative) part $\U_X^{\trho,+}$ (resp. $\U_X^{\trho,-}$) by requiring in the above definition that $\x\in\Z^+$ (resp. $\x\in\Z^-$).
\end{defin}

Observe that if $n=1$ and $\trho=1\in\wXF{}$ the trivial character, then  the above definition specializes to the spherical Hall algebra $\U_X$ as defined in \cite{BS} Section 4 or as considered implicitly in \cite{Kap} Section 3.8 and Section 5.

\subsection{}\label{ss:presentation twisted spherical}
In this subsection we recall the combinatorial description of the spherical Hall algebra as discovered in \cite{BS} Section 5.

We need to introduce first some notations. For two points $\x,\y\in\Z^*$ which are not proportional we denote by $\epsilon_{\x,\y}:=\mathrm{sign}(\det(\x,\y))\in\{\pm1\}$ and by $\Delta_{\x,\y}$ the triangle formed by the vectors $\o,\x,\x+\y$ where $\o=(0,0)$ is the origin.

We arrived now at the definition of the universal spherical Hall algebra (see also \cite{BS} Section 6).
\begin{defin}\label{D:entire form} Fix $\sigma,\ov\sigma\in\C^*$ with $\sigma,\ov\sigma\not\in\{\pm1\}$ and set $\nu:=(\sigma\ov\sigma)^{-1/2}$ and
\[c_{i}:=(\sigma^{i/2}-\sigma^{-i/2})(\ov\sigma^{i/2}-\ov\sigma^{-i/2})[i]_\nu/i.\]

Let $\E_{\sigma,\ov\sigma}^n$ to be the $\C$-algebra generated by $\{t_\x\mid \x\in\Z^*\}$ modulo the following relations:
\begin{enumerate}
	\item If $\x,\x',\o$ are collinear then \[[t_\x,t_{\x'}]=0\]
	\item If $\x,\y$ are such that $\delta(\x)=1$ and that $\Delta_{\x,\y}$ has no interior lattice point then
	\[[t_\y,t_\x]=\epsilon_{\x,\y}c_{n\delta(\y)}\frac{\theta_{\x+\y}}{n(\nu^{-1}-\nu)}\]
	where the elements $\theta_\z$ are defined by equating the coefficients of the following two series:
	
	\[\sum_{i\ge 0}\theta_{i\z_0}s^i=\exp\left(n(\nu^{-1}-\nu)\sum_{i\ge 1}t_{i\x_0}s^i\right)\]
	for any $\x_0\in\Z^*$ such that $\delta(\x_0)=1$.
\end{enumerate}
\end{defin}

The algebra $\E_{\sigma,\ov\sigma}^n$ comes equiped with a natural $\SL_2(\ZZ)$ action given by $\gamma\cdot t_{\x} = t_{\gamma\cdot\x}$ where $\gamma\in\SL_2(\ZZ)$ and $\x\in\Z^*$.
\vspace{0.1in}

Let us give a more detailed description of the algebra $\UU_{\s,\ov\s}^n$ following Section 5 of~\cite{BS}. 

Recall that that we denoted by $\Z=\ZZ^2$ and by $\o=(0,0)$ the origin of $\Z$. By a path in $\Z$ we understand a sequence $\p=(\x_1,\dots,\x_r)$ of non-zero elements of $\Z$ which we represent graphically as the polygonal line in $\Z$ that joins the points $\o,\x_1,\x_1+\x_2,\dots,\x_1+\dots+\x_r$. Let $\widehat{\x\y}\in [0,2\pi)$ denote the angle between the segments $\o\x$ and $\o\y$. We will call a path $\p=(\x_1,\dots,\x_r)$ convex if $\widehat{\x_1\x_2}\le\widehat{\x_1\x_3}\le\dots\le\widehat{\x_1\x_r}<2\pi$. Put $L_0:=\mathbb{N}(0,-1)$ and let $\Conv'$ be the collection of all convex paths $\p=(\x_1,\dots,\x_r)$ satisfying $\widehat{\x_1L_0}\ge\dots\ge\widehat{\x_r L_0}$. Two convex paths $\p$ and $\q$ in $\Conv'$ are said to be equivalent if $\p$ is obtained by permuting several segments of $\q$ of the same slope. For example the path $((0,1),(0,2),(1,3))$ is equivalent to the path $((0,2),(0,1),(1,3))$. We denote by $\Conv$ the set of equivalence classes of paths in $\Conv'$ and we will call the elements of $\Conv$ simply paths. We introduce the positive paths $\Conv^+$ and the negative paths $\Conv^-$ as the paths $\p=(\x_1,\dots,\x_r)\in\Conv$ such that $\widehat{\x_r L_0}\ge\pi$ respectively $\pi>\widehat{\x_1 L_0}$. By concatenating paths we obtain an identification $\Conv \equiv \Conv^+\times\Conv^-$.

Fix an integer $n\ge 1$. To a path $\p=(\x_1,\dots,x_r)\in\Conv$ we associate the element $t_{\p}\in\UU_{\sigma,\ov\sigma}^n$ defined by
\[
 t_\p:=t_{\x_1}\cdot\dots\cdot t_{\x_r}.
\]
Observe that this is a well defined element of $\UU_{\sigma,\ov\sigma}^n$ due to the relation (1).

In~\cite{BS} Lemma 5.6 it is proved that $\{t_\p\mid \p\in\Conv\}$ is a $\C$-basis for the algebra $\UU_{\sigma,\ov\sigma}^1$ when $\sigma,\ov\sigma$ are specialized to be the eigenvalues of the Frobenius on the first \'etale cohomology group of $X$.

With the above choice of $\sigma$ and $\ov\sigma$ one of the main results of \cite{BS} is:

\begin{thm}[\cite{BS} Theorem 5.4] The assignment $\Omega:\t_\x\mapsto T_\x$ for $\x\in\Z$ extends to an algebra isomorphism 
\[\Omega:\E_{\sigma,\ov\sigma}^1\to \U_X\ot_K\C.\]
\end{thm}

This isomorphism obviously intertwines the action of $\SL_2(\ZZ)$ on the two algebras.

\vspace{0.1in}

One of the goals of this article is to extend the above result to all the twisted spherical Hall algebras and use it to study the cusp eigenforms on $X$. It will turn out that each algebra $\U_X^\trho$ is in fact the subalgebra generated by the coefficients of a cusp eigenform and some torsion sheaves naturally associated to this cusp form. As a byproduct we obtain a description of the entire Hall algebra $\HH_X$: it is isomorphic to the infinite\footnote{restricted and commutative} tensor product of all the twisted spherical Hall algebras.

\section{Automorphic forms and Rankin-Selberg $L$-functions}

We will review here the basic setup of (non-ramified) automorphic forms and recall the definition of Rankin-Selberg $L$-functions. This Section is based on the Sections 2 and 3 of \cite{Kap}. We invite the reader to take a look at [op.cit.] for a more thorough discussion. Throughout we will only work over an elliptic curve but everything makes sense for an arbitrary smooth projective curve.

\vspace{0.1in}

\subsection{} We denote by $K_X$ the field of rational functions on $X$, by $\AA_X$ the ring of ad\`eles on $X$ and by $\OO_X$ the integer ad\`eles.

By a theorem of A. Weil we know that the set $\Bun_n(X)$ can be identified with the double quotient $\GL_n(K_X)\backslash\GL_n(\AA_X)/\GL_n(\OO_X)$.

A non-ramified automorphic form of rank $n$ on $X$ is a $\C$-valued function on $\Bun_n(X)$. Since we only deal with non-ramified automorphic forms we will suppress the adjective "non-ramified". 
We denote the vector space of automorphic forms of rank $n$ on $X$ by $\AF_n$. Observe that the automorphic forms with finite support are elements of the Hall algebra $\HH_X$.

\vspace{0.1in}

We say that an automorphic form $f$ of rank $n$ is a cusp form if for any proper parabolic subgroup $P\le \GL_n$ and for any $g\in \GL_n(\AA_X)$ we have
\[
 \int_{U_P(K_X)\backslash U_P(\AA_X)} f(ug)du=0
\]
where $U_P$ is the unipotent radical of $P$.

In the language of Hall algebras this definition can be restated as follows:
an automorphic form $f$ of rank $n$ is a cusp form if for any non-zero fiber bundles $\kV,\kW$ of ranks $n',n''<n$ we have that 
\[\sum_{\xi\in\Ext(\kV,\kW)}f(\Cone(\xi)[-1])=0\]
where the notations are the same as for the Hall algebra.

\vspace{0.1in}

Let $\AFc_n$ denote the set of cusp forms of rank $n$ and by $\AFc_{n,d}$ the set of cusp forms of rank $n$ and degree $d$.
We have the following important proposition:

\begin{prop} Every function $f\in\AFc_{n,d}$ has finite support and the space $\AFc_{n,d}$ is finite dimensional.
\end{prop}
This result for an arbitrary curve $X$ is a consequence of Harder's reduction theory. See for example \cite{Ha1,Ha2}. However, in our case, $X$ an elliptic curve, the proof is not difficult and we can sketch it here. It is easy to see that the cusp forms are supported on the semistable sheaves (because the Harder-Narasimhan filtration splits and this in turn because the canonical sheaf of $X$ is trivial) and from Atiyah's theorem we have a classification of these. In particular, for integers $n\ge 1, d\in \ZZ$ we know that there are only a finite number of semistable sheaves of rank $n$ and degree $d$. This forces the vector space $\AFc_{n,d}$ to be finite dimensional and any cusp form to have finite support.
\vspace{0.1in}

\subsection{} Let $\tau\in\Tor_X$ be a torsion sheaf and $n\ge 1$. The Hecke operator  $\He_\tau:\AF_n\to\AF_n$ associated to $\tau$ is defined by:
\[
 \He_\tau(f)(\kV) = \sum_{\substack{\kV'\subseteq \kV\\\kV/\kV'\simeq\tau}}f(\kV').
\] 
where $f\in\AF_n$.

Observe that this is the same definition as in the Hall algebra context.

We have an action of the (commutative) Hecke algebra $\HH_X^\tor$ on the vector space $\AF_n$. It is not difficult to see that the Hecke operators preserve in fact the space of cusp forms $\AFc_n$. Moreover the multiplicity one theorem of Shalika and Piatetski-Shapiro (\cite{P-S, Sh}) tells us how this action decomposes into irreducibles:

\begin{thm}[\cite{P-S, Sh}]
The action of the Hecke algebra $\HH_X^\tor$ on the space $\AFc_n$ is diagonalisable and every eigenspace appears with multiplicity 1.
\end{thm}

\begin{defin} An automorphic form $f\in \AFc_n$ is called a cusp eigenform if $f$ is an eigenvector for all the Hecke operators.
\end{defin}

The multiplicity one theorem says that the cusp eigenforms span the space of cusp forms and that they have different eigenvalues under the Hecke operators.

\subsection{}\label{def Rankin Selberg} 
Let $f\in\AFc_n$ be a cusp eigenform. For $x\in X$ a closed point and $i=1,\dots,n$ we denote by $z_{i,x}(f)$ the unique (up to permutation) complex numbers that verify for every $l=1,\dots,n$
\[
 \He_{\O_x^{\oplus l}}(f) = q_x^{l(n-l)/2}e_l(z_{1,x}(f),\dots,z_{n,x}(f))f
\]
where $e_l$ is the $l$-th elementary symmetric function. The numbers $z_{i,x}(f),i=1,\dots,n, x\in X$ are called the Hecke eigenvalues of $f$.

If $f$ and $g$ are two cusp eigenforms of rank $n$ and $m$ respectively then their Rankin-Selberg $L$-function is defined by the formula:
\[
L(f,g,t):=\prod_{x\in X}\prod_{i=1}^n\prod_{j=1}^m(1-z_{i,x}(f)^{-1}z_{j,x}(g)t^{|x|})^{-1}.
\]

We summarize their most important properties in the following theorem:

\begin{thm}[see \cite{JPS}] Let $f$ and $g$ be cusp eigenforms. If $f\neq g$ then $L(f,g,t)=1$ and if $f=g$ then $L(f,f,t)$ converges (for $|t|<q^{-1}$) to a rational function.
\end{thm}
\begin{rem} The analogue result for an arbitrary curve $X$ is that if $f\neq g$ then $L(f,g,t)$ converges to a polynomial and if $f=g$ then $L(f,f,t)$ converges to a rational function.
\end{rem}

We will need a more precise statement about the function $L(f,f,t)$, namely:
\begin{prop}\label{P:value L func} If $f$ is a cusp eigenform of rank $n$ then \[L(f,f,t)=\zeta_{X_n}(t^n)\]
where the equality should be understood in the sense that the $L$-function converges in some region to the rational function defined on the right hand side.
\end{prop}

This result is probably well known but we couldn't find a reference for it and hence we included an ad-hoc proof that uses the precise determination of the Hecke eigenvalues. We defer the proof to Section \ref{S:computations Rankin-Selberg}.

\subsection{} Let $f$ be a cusp eigenform of rank $n$. Define its associated generating series with coefficients in $\HH_X^\vvec$ to be:
\[
 E_f(z):=\sum_{\V\in\Bun_n} f(\V)[\V]z^{\deg{\V}}.
\]

If $\chi_f:\HH_X^\tor\to\C$ is the eigenvalue of $f$ then Kapranov considered in \cite{Kap} the following generating series:
\[
 \psi_f(z):=\sum_{\tau\in\Tor_X}\ov{\chi_f}([\tau])a_{\tau} z^{|\tau|}[\tau]\in \HH_X^\tor[[z]]
\]
where for a torsion sheaf $\tau\in\Tor_X$ we denoted by $a_\tau = \#\Aut(\tau)$ and by $|\tau|$ the degree of $\tau$.

Observe that we cannot multiply $E_f(z)$ by $E_g(z)$ or by $\psi_g(z)$ for two cusp eigenforms because of the infinite summands in the coefficients. However, we can multiply $E_f(z_1)$ by $E_g(z_2)$ or by $\psi_g(z_2)$.

\begin{thm}[\cite{Kap}, Theorem 3.3]\label{T:rel E_f,E_g,psi_f} Let $f$ and $g$ be two cusp eigenforms of rank $n$ respectively $m$. Then we have the following commutation relations:
 \begin{enumerate}
  \item For any coherent sheaf $\kF\in\Coh(X)$ the coefficient of $[\kF]$ in the products $E_f(z_1)\cdot E_g(z_2)$, $E_f(z_1)\cdot\psi_g(z_2)$ and $\psi_f(z_1)\cdot E_g(z_2)$ is a power series in $z_1,z_2$ which converges for $|z_1|>>|z_2|$ to a rational function.
  \item These rational functions satisfy the following functional equations 
\begin{equation}\label{E:commutation E_f and E_g}
E_f(z_1)\cdot E_g(z_2) = \ds\frac{L(f,g,z_2/qz_1)}{L(f,g,z_2/z_1)}E_g(z_2)\cdot E_f(z_1) 
\end{equation}

\begin{equation}\label{E:commutation E_f and psi_g}
E_f(z_1)\cdot \psi_g(z_2) = \ds\frac{L(f,g,q^{m/2-1}z_2/z_1)}{L(f,g,q^{m/2}z_2/z_1)}\psi_g(z_2)\cdot E_f(z_1).   
\end{equation}
  \item 
\[
\Delta(\psi_f(z)) = \psi_f(z)\otimes\psi_f(z)
\]

\[
 \Delta(E_f(z)) = E_f(z)\ot 1 +\psi_f(q^{-n/2}z)\ot E_f(z)
\]
\end{enumerate}
\end{thm}

\begin{rem}
The above equalities should not be understood as equalities of formal power series. Their meaning is that whenever we evaluate the terms at some coherent sheaf $\kF\in\Coh(X)$ the resulting power series will be convergent (in different regions!) to the same rational function.
\end{rem}

\begin{rem}
A word of warning: we use the (co)product opposite to the one used by Kapranov in~\cite{Kap} and we do not need to consider the extended Hall algebra because the symmetric Euler form is trivial. This is why, at the first sight, the above formulas might look different from the ones in~[loc.cit.].
\end{rem}

Let us write $E_f(z) = \sum_{d\in\ZZ} E_{f,d}z^d$.
Using the fact that the Rankin-Selberg $L$-functions converge to rational functions we can clear out the denominators in the formula \ref{E:commutation E_f and E_g} by multiplying with appropriate polynomials. Kapranov proved (see~\cite{Kap}, Theorem 3.5.6) that in fact this new formula can be interpreted as an equality of formal power series and comparing the coefficients of $z_1^iz_2^j$ for $i,j\in\ZZ$ gives valid relations between the elements $E_{f,d},E_{g,e}\in \HH_X$ for $d,e\in\ZZ$. 

\subsection{}
In \cite{Kap}, for a cusp eigenform $f$ of rank $n$, the following power series was considered:
\[
 a_f(z) = \sum_{d\ge 1}a_{f,d}z^d = \log(\psi_f(z)).
\]

In a similar way we obtain from the formula \ref{E:commutation E_f and psi_g} valid relations between the elements $E_{f,d},a_{g,d}\in \HH_X$.

\vspace{0.1in}

Let us consider the subalgebra $\HH_X^\circ$ of $\HH_X$ generated by $E_{f,d}$ and $a_{f,e}$ where $f$ runs over the cusp eigenforms and $d\in\ZZ,e\in\ZZ_+$.

Kapranov asked in~\cite{Kap} if the above relations suffice to give a presentation of $\HH_X^\circ$. More precisely, let $\widetilde\HH$ be an algebra generated by the symbols $\tilde E_{f,d},\tilde a_{f,e}$ where $f$ runs over the cusp eigenforms and $d\in\ZZ,e\in\ZZ_+$ subject to the relations obtained from Theorem~\ref{T:rel E_f,E_g,psi_f} as described above. We have an obvious surjective algebra morphism $\pi_X:\widetilde\HH\to\HH_X^\circ$. The problem is to determine its kernel. The presumably new relations satisfied by $E_{f,d}$ are also called in the literature higher rank relations between the residues of the Eisenstein series. 

If $X$ is the projective line, in \cite{Kap} it is proved that $\pi_X$ is an isomorphism and moreover that this presentation of $\HH_X^\circ$ gives an isomorphism of $\HH_X^\circ$ with a certain positive part of the quantum loop group $U_\nu(\widehat{\sl_2})$.

\vspace{0.1in}

For $X$ an elliptic curve O. Schiffmann has addressed in~\cite{S4} the similar problem for the spherical Hall algebra. Namely, for the cusp eigenform 
\[
f=\sum_{d\in\ZZ}\sum_{\kL\in\Pic^d(X)}[\kL]
\]
he considered the subalgebra $\U_X^+$ of $\HH_X$ generated by $E_{f,d},a_{f,e},d\in\ZZ,e\in\ZZ_+$ 
and the (abstract) algebra $\widetilde\UU$ generated by the symbols $\tilde E_{f,d},\tilde a_{f,e}$ subject to the relations deduced from the functional equations~\ref{E:commutation E_f and E_g} and \ref{E:commutation E_f and psi_g}. He proved that the kernel of the map $\pi_X:\widetilde\UU\to\U_X^+$ is generated by the cubic relations:
\[
 [[E_{f,l+1},E_{f,l-1}],E_{f,l}]=0,\, \forall l\in\ZZ.
\]

As a corollary of our main results combined with~\cite{S4} we obtain the same description of the higher rank relations satisfied by the residues of the Eisenstein series for all the cusp eigenforms.

\subsection{}\label{SS:strength mult 1}  In~\cite{Kap} Section 3.8 it is conjectured that the elements $a_{f,d}$, where $f$ is a cusp eigenform and $d\in\ZZ_+$, are algebraically independent over $\C$. This conjecture can be viewed as a certain strengthening of the multiplicity one theorem.

We can give a positive answer to this question using the description of the cusp eigenforms and of their coproduct. 

\section{The main results}
\subsection{}
In order to state our main results we need to introduce another family of automorphic forms. Namely, for $n\ge 1$ and for $\trho\in\wXF{n}$ we define 
\[
T_{n}^{\tilde\rho}:=\sum_{d\in\ZZ}T_{(n,nd)}^{\tilde\rho}.
\]

The following theorem gives the structure of the space of cusp forms on $X$:

\begin{thm}\label{T:cusp forms} For any integer $n\ge1$ we have:
\begin{enumerate}
\item The space $\AFc_{n,d}$ is zero unless $n|d$.
\item A basis for $\AFc_{n,dn}$ is given by $\{T_{n,dn}^{\tilde\rho},\tilde\rho\in\P_n\}$
\item For any primitive character $\tilde\rho\in\P_n$ the automorphic form $T_{n}^{\tilde\rho}$ is a cusp eigenform.
\end{enumerate}
\end{thm}

For the structure of the twisted spherical Hall algebras and for the whole Hall algebra we have the following result:

\begin{thm}\label{T:str of Hall} For any integers $n,m\ge 1$ and any two different primitive characters $\trho\in\P_n$ and $\tsigma\in\P_m$ we have:
\begin{enumerate} 
 \item The twisted spherical Hall algebra $\U_X^{\tilde\rho}$ is isomorphic to $\E_{\sigma,\ov\sigma}^n$
 \item The algebras $\U_X^{\tilde\rho}$ and $\U_X^{\tilde\sigma}$ commute with each other 
 \item The Hall algebra $\HH_X$ decomposes into a commutative (restricted) tensor product: 
\[
 \HH_X\simeq \underset{\tilde\rho\in\P}{{\bigotimes}'} \U_X^{\tilde\rho,+}
\]
(isomorphism of bialgebras). In particular we have that
\[
\DH_X\simeq  \underset{\tilde\rho\in\P}{{\bigotimes}'} \U_X^{\tilde\rho}
\]
(isomorphism of algebras).

\end{enumerate}
\end{thm}

The proofs (see Section~\ref{S:proof main results}) will be by induction on $n$ and we will deal with both theorems at the same time: using the result of Theorem~\ref{T:str of Hall} for $r<n$ we prove Theorem~\ref{T:cusp forms} for $r=n$ and then we proceed to prove the case $r=n$ of Theorem~\ref{T:str of Hall}, etc.

As an easy corollary of the proof of these theorems (see Step $3_n$ in Section~\ref{S:proof main results}) we get the following corollary (see Section~\ref{SS:strength mult 1} and~\cite{Kap} Conjecture 3.8.5):
\begin{cor}\label{C:strength mult 1} The (nonzero) elements $a_{f,d}$, where $f$ runs over the set of cusp eigenforms\footnote{as usual we only consider the cusp eigenforms up to the $\C^\times$ action} and $d\ge 1$, are algebraically independent.
\end{cor}
\begin{rem}
A proof of a reformulation of this result (see~\cite{Kap} Reformulation (3.8.6)) can be given using Proposition~\ref{P:value L func}.
\end{rem}

\subsection{} The statement about the higher functional equations for the Eisenstein series requires a few more preparations.

Put $\chi_n(z,w) = (z-\sigma^n w)(z-\overline{\sigma}^n w)(z-(\sigma\overline{\sigma})^{-n} w)$ and $\chi_{-n}(z,w) = -\chi_n(w,z)$.  

Following ~\cite{S4} we consider the formal power series 
\[
\TT_1(z) = \sum_{d\in\ZZ} u_{(1,d)}z^d
\]
\[
 \TT_0^+(z) = 1+\sum_{l\ge 1} \theta_{(0,l)}z^l
\]
For a formal power series $A = \sum_{d\in\ZZ} a_d z^d$ the residue operator is defined by $\Res_z(A) = a_{-1}$. We define $\Res_{z,y,w}$ to be a successive application of residue operators with respect to the variables $w,y,z$.

\begin{defin}
 Let $\tUU^n_{\sigma,\ov\sigma}$ be the algebra generated by the Fourier coefficients of $\TT_1(z)$ and $\TT_0^+(z)$ subject to the relations:
\[
 \TT_0^+(z)\TT_0^+(w) = \TT_0^+(w)\TT_0^+(z)
\]
\[
 \chi_n(z,w)\TT_0^+(z)\TT_1(w) = \chi_{-n}(z,w)\TT_1(w)\TT_0^+(z)
\]
\[
 \chi_n(z,w)\TT_1(z)\TT_1(w) = \chi_{-n}(z,w)\TT_1(w)\TT_1(z)
\]
\[
 \mathrm{Res}_{z,y,w}[(zyw)^m(z+w)(y^2-zw)\TT_1(z)\TT_1(y)\TT_1(w)] = 0,\,\forall m\in\ZZ.
\]
\end{defin}
Put $\alpha_i:=(1-\sigma^i)(1-\overline{\sigma}^i)(1-(\sigma\overline{\sigma})^{-i})/i$. Introduce the elements $u_{(0,i)}, i\ge 1,$ via the following formula:
\[
 \log(\TT_0^+(z)) = \sum_{i\ge 1} n\alpha_{ni}u_{(0,i)} z^i
\]

\begin{thm}\label{T:higher Eis rels} Let $\trho\in\P_n$. 
The map $\Psi:\tUU^n_{\sigma,\ov\sigma} \to \U^{\trho,+}$ defined by 
\[
 u_{(1,i)}\mapsto (v^{-1}-v)\alpha_{n}^{-1} T_{(n,ni)}^{\trho}
\]
\[
 u_{(0,i)}\mapsto (v^{-1}-v)\alpha_{ni}^{-1} T_{(0,ni)}^{\trho}
\]
is an isomorphism of algebras.
\end{thm}
\begin{proof}
It follows by unwinding the definitions and by using Theorem~\ref{T:str of Hall} (1), Theorem~\ref{T:rel E_f,E_g,psi_f} (2), Proposition~\ref{P:value L func} and \cite{S4} Theorem 4.
\end{proof}

\section{Frobenius eigenvalues and actions of Hecke operators}
In this section we compute the Frobenius eigenvalues for an irreducible $l$-adic representation of the fundamental group of $X$ and the action of the Hecke operators on the corresponding cusp eigenforms. These results together with the computations from the next section form the technical core of the paper.

\subsection{Frobenius eigenvalues}\label{S:Langlands correspondence}

We will work out in this subsection the $l$-adic representations involved in the unramified Langlands correspondence for an elliptic curve. We will determine the Frobenius eigenvalues in terms of some character of a Picard group.

Let us first outline the big lines that we will follow below to determine de Frobenius eigenvalues.

So let $V$ be an irreducible $l$-adic representation of $\pi_1(X)$ of dimension $n\ge 1$. Since $X$ is an elliptic curve its geometric fundamental group $\pi_1(\ov X)$ is abelian and hence the restriction of $V$ to $\pi_1(\ov X)$ is a sum of $n$ characters. Now the Galois group $\Gal(\ov\ff_q/\ff_q)$ permutes transitively these characters. It follows that the Galois group $\Gal(\ov\ff_q/\ff_{q^n})$ acts trivially on them and hence the restriction of $V$ to $\pi_1(X_n)$ is also a sum of characters and moreover the representation $V$ is an induced representation from a character, say $\rho$, of $\pi_1(X_n)$. From abelian class field theory we know that the characters of $\pi_1(X_n)$ are the same as the characters of $\Pic(X_n)$ and hence the irreducible representations of $\pi_1(X)$ of dimension $n$ are classified by (some) characters of the Picard group $\Pic(X_n)$. Denote by $\rho'$ the character of $\Pic(X_n)$ that corresponds to $\rho$ by class field theory. Using the above description of $V$ as an induced representation from $\rho$, we can determine the eigenvalues of $\Frob_x$ (the Frobenius conjugacy class associated to $x$) in terms of $\rho'$ and the points of $X_n$ that sit over $x$.

\vspace{0.1in}

Recall that we denoted  by $X_k$ the extension of $X$ to $\ff_{q^k}$ and by $\overline X$ the extension of $X$ to $\overline\ff_q$.

For every positive integer $k\ge 1$ denote by $G_k = \pi_1(X_k)$ the algebraic \'etale fundamental group and by $\overline{G} = \pi_1(\overline{X})=\pi_1^{\mathrm{geom}}(X)$ the geometric \'etale fundamental group. We also set $G:=G_1$. We have exact sequences of groups
\[
1\to \ov G\to G\to \pi_1(\ff_q)=\widehat\ZZ\to 1
\]
\[
1\to G_k\to G \to \Gal(\ff_{q^k}/\ff_q)=\ZZ/k\ZZ\to 1
\]
where we view the abelian groups $\widehat\ZZ$ and $\ZZ/k\ZZ$ multiplicatively.
The group $\hat\ZZ$ is generated (topologically) by the absolute Frobenius automorphism of the field $\ov\ff_q: t\mapsto t^q$.

\begin{rem} A crucial remark is that for an elliptic curve the geometric fundamental group $\pi_1(\overline{X})$ is abelian.
\end{rem}

Let $(V,\rho)$ be an irreducible $l$-adic representation of $G$ of dimension $n$. By this we mean that $V$ is a finite dimensional vector space over $\ov\Q_l$ and $\rho:G\to \GL(V)$ is a continuous morphism where $V$ is given the $l$-adic topology. Moreover we require the morphism $\rho$ to be actually defined over some finite extension of $\Q_l$.

Since $\ov G$ is an abelian group, the restriction of the representation $V$ to $\overline{G}$ is a sum of characters, say $V|_{\ov G}=\rho_{(1)}\oplus\dots\oplus\rho_{(n)}$ (see Lemma~\ref{L:restr of rep over Gbar are semisimple}). The (Galois) group $\widehat\ZZ$ permutes these characters and since $V$ is irreducible it permutes them transitively. This means that the action of $\widehat\ZZ$ factorizes through the finite quotient $\ZZ/n\ZZ$ and therefore the representation $V$ restricted to $G_n$ will also be a sum of characters. For convenience we will also denote them by $\rho_{(1)},...,\rho_{(n)}$. 
 
It follows from the above that $V\simeq \Ind_{G_n}^G(\rho_{(1)})$ (see for example \cite{Serre} Ch. 7 Prop. 19).

Let $x\in X$ be a closed point. Attached to this point we have a well defined conjugacy class\footnote{give a short short definition of it; add class field th} $\Frob_x$ in $G$. We are interested in knowing the eigenvalues of this conjugacy class acting on the representation $V$.

Denote by $f=|x|$ the degree of $x$ and put $d=\gcd(n,f)$ and  $m=\lcm(n,f)$.

Choose a point $x'\in X_d$ that sits above $x\in X$ and a point $x''\in X_n$ that sits above $x'\in X_d$. 
We have that $k(x)=\ff_{q^f}$ (by the definition of the degree), $k(x') = \ff_{q^f}$ and $k(x'') = \ff_{q^m}$.

Associated to these three points we have three Frobenius conjugacy classes: $\Frob_x\in G,\Frob_{x'}\in G_d,\Frob_{x''}\in G_n$.

We know from \cite{CF} (pag. 166, Prop. 3.2) that $\Frob_x = \Frob_{x'}$ and that $\Frob_{x'}^{n/d} = \Frob_{x'}^{m/f} = \Frob_{x''}$ where the equalities are understood as equalities of conjugacy classes.

Let us give names to the maps we are going to work with:
\[
1\to G_n\to G\stackrel{p}{\to}\ZZ/n\ZZ = \{1,\f,...,\f^{n-1}\}\to 1
\]

\[
1\to G_d\to G \stackrel{p_1}{\to}\ZZ/d\ZZ = \{1,\ov\f,\ov\f^2,...,\ov\f^{d-1}\}\to 1
\]

\[
1\to G_n\to G_d\stackrel{p_2}{\to} \ZZ/(n/d)\ZZ=\{1,\f^d,\f^{2d},...,\f^{n-d}\}\to 1
\]

\begin{lemma} Let $x\in X$ be a point of degree $d'$. Then $\Frob_x\not\in G_n$ for any $n>1$ such that $\gcd(n,d')=1$.
\end{lemma}
\begin{proof} Let $Y$ be a finite, \'etale cover of $X_n$. We will prove that $\Frob_x\not\in\Aut_{X_n}(Y)$. Let $x'\in X_n$ be a point above $x$ and let $y\in Y$ be a point above $x'$.
By definition $\Frob_x\in\Aut_X(Y)$ and is the (canonical) generator of $\Gal(k(y)/k(x))$. We have that $\Gal(k(x')/k(x))$ has degree $n>1$ because $\gcd(n,d')=1$. Therefore $\Frob_x\in\Gal(k(y)/k(x))$ will not fix the field $k(x')$, so it can not be in the group $\Aut_{X_n}(Y)$. By passing to the limit we see that $\Frob_x\not\in\pi_1(X_n)$.
\end{proof}

\begin{rem}In the above proof we used that the cover was not ramified in order to pass back and forth from $\Aut_X(Y)$ to $\Gal(k(y)/k(x))$.\end{rem}

\begin{lemma}\label{lem_frob_power} Let $x\in X$ be a point of degree $f$ and let $n$ be a positive integer. Denote by $d=\gcd(f,n)$. Then the image of $\Frob_x$ by the map $\pi_1(X)\to \ZZ/n\ZZ=\Gal(\ff_{q^n}/\ff_q)$ is the $d$-th power of a generator.
\end{lemma}

\begin{proof} We will need to use the following commutative diagram:
\[
\xymatrix{
& \ZZ/d\ZZ\ar[r]^{=} & \ZZ/d\ZZ \ar[r]^{=\qquad} & \{1,\ov\f,\ov\f^2,...,\ov\f^{d-1}\}\\
G_n\ar[r] & G\ar[r]\ar[u] & \ZZ/n\ZZ\ar[u]\ar[r]^{=\qquad} & \{1,\f,\f^2,...,\f^{n-1}\}\\
G_n\ar[u]^{=}\ar[r] & G_d\ar[u]\ar[r] & \ZZ/(n/d)\ZZ\ar[u] \ar[r]^{=\qquad} & \{1,\f^d,...,\f^{(n-1)d}\}\\
}\]
We know that $\Frob_x\in\pi_1(X_d)$. Now, since $\gcd(f/d,n/d)=1$ we are in the situation of the previous lemma and it follows that $\Frob_x\not\in G_{dh}$ for any $h$ such that $\gcd(h,n/d)=1$. This implies that the image of $\Frob_x$ by the map $G_d\stackrel{p_2}{\to} \ZZ/(n/d)\ZZ$ is a generator, hence the image of $\Frob_x$ by the map $G\stackrel{p}{\to} \ZZ/n\ZZ$ is the $d$-th power of a generator.
\end{proof}

Now let us continue with the determination of the Frobenius eigenvalues.
As $\Frob_x=\Frob_x'\in G_d$ we see that $p_1(\Frob_x) = 1$.
Moreover, since $p(\Frob_x)\in \ZZ/n\ZZ$ is the $d$-th power of a generator (Lemma~\ref{lem_frob_power}), we can suppose that $p(\Frob_x) = \f^d$. Therefore $p_2(\Frob_x) = \f^d$.
The groups $G_n,G_d$ fit in a short exact sequence of the form:
\[
1\to G_n\to G_d\stackrel{p_2}{\to} \ZZ/(n/d)=\{1,\f^d,\dots,\f^{n-d}\}\to 1
\]
and it is clear that $G_n$ together with $\Frob_x$ generate the group $G_d$.

By the transitivity of the induction we can write $V = \Ind_{G_d}^G(\Ind_{G_n}^{G_d}(\rho_1))$.
Let's denote by $V_d:=\Ind_{G_n}^{G_d}(\rho_1)$. We will first describe the action of $\Frob_x=\Frob_{x'}$ on this representation.

Since $p_2(\Frob_x)=\f^d$ we can write down explicitly the module $V_d$. As a vector space it is given by $\qqb_lv\oplus \qqb_l\Frob_x v\oplus \qqb_l \Frob_x^2v\oplus...\oplus \qqb_l\Frob_x^{n/d-1}v$ where the vectors $\Frob_x^iv$,  $i=0,...,n/d-1$ are a basis of $V_d$ and $v$ is a basis for the 1 dimensional representation $\rho_{(1)}$.
The module structure is given as follows: if $h\in G_n$ then $h\Frob_x^j\cdot \Frob_x^iv = \rho_1(\Frob_x^{-i-j}h\Frob_x^{i+j})\Frob_x^{i+j} v$ and $\Frob_x^{n/d}v = \rho_{(1)}(\Frob_{x''})v$.
 
From the above description we obtain that the element $\Frob_x$ acts on $V_d$ by the following $n/d\times n/d$ matrix (in the chosen basis):
\[
D_1:=\begin{pmatrix} 
0 & 1 & 0 & 0 & ... & 0\\
0 & 0 & 1 & 0 & ... & 0\\
\vdots\\
0 & 0 & 0 & 0 & ... & 1\\
\rho_1(\Frob_{x''}) & 0 & 0 & 0 & ... & 0\\
\end{pmatrix}
.\]

Let us now proceed to the determination of the action of $\Frob_x$ on $V$.

Remember that we chose a rational point $x_0$ on the curve $X$. The Frobenius $\Frob_{x_0}\in G$ is sent (see Lemma \ref{lem_frob_power}) by the map $p$ to $\f\in\ZZ/n\ZZ$. It is clear from Galois theory that $\Frob_{x_0}$ permutes transitively the points of $X_n$ that sit above $x$ and hence it permutes (by conjugation) the Frobenius elements associated to these points. Namely, let $x_1',\dots,x_d'\in X_d$ be the points of $X_d$ that sit above $x$ and let $x_1'',\dots,x_d''\in X_n$ be the points that sit over $x_1',\dots,x_d'$ (observe here that there's only one point of $X_n$ that sits over each $x_i'\in X_d$ ).

The Frobenius $\Frob_{x_0}$ permutes transitively the sets $\{x'_i,i=1,\dots,d\}$, $\{x_i'', i=1,\dots,d\}$ and hence by conjugation it also permutes transitively the set $\{\Frob_{x_i''}, i=1,\dots,d\}$. By relabeling if necessary we can suppose that $\Frob_{x_0}$ permutes the set $\{\Frob_{x_i''}, i=1,\dots,d\}$ in the order $\Frob_{x_1''}\to \Frob_{x_2''}\to\dots\to\Frob_{x_d''}\to\Frob_{x_1''}$.

Now we are ready to describe explicitly the representation $V$. As a vector space we can write it $V=V_d\oplus \Frob_{x_0}V\oplus...\oplus\Frob_{x_0}^{d-1}V$. The action of $G$ is given as follows:
if $h\Frob_{x_0}^j$ is a general element of $G$, where $h\in G_d$, and if $w\in V_d$ is an arbitrary vector then
$h\Frob_{x_0}^j\cdot \Frob_{x_0}^iw = \Frob_{x_0}^{i+j}(\Frob_{x_0}^{-i-j}h\Frob_{x_0}^{i+j})\cdot w$ where the last $\cdot$ means the action of $G_d$ on $V_d$. Observe that $\Frob_{x_0}^d\in G_d$ and therefore it acts on $V_d$.

In particular, the action of $\Frob_x$ on $V$ is given by the following block matrix (in the above chosen basis):
\[
\begin{pmatrix}
D_1 & 0 & 0 & ...& 0\\
0 & D_2 & 0 & ...& 0\\
\vdots\\
0 & 0 & 0 & ... & D_d
\end{pmatrix}
\]
where each block-matrix $D_i$ is of the form:
\[D_i =\begin{pmatrix}
0 & 1 & 0 & 0 & ... & 0\\
0 & 0 & 1 & 0 & ... & 0\\
\vdots\\
0 & 0 & 0 & 0 & ... & 1\\
\rho_{(1)}(\Frob_{x_i''}) & 0 & 0 & 0 & ... & 0\\
\end{pmatrix}
\]
It is easy to see now that the characteristic polynomial of $\Frob_x$ acting on $V$ is given by $\prod_{i=1}^d(T^{n/d}-\rho_{(1)}(\Frob_{x_i''}))$.

Class field theory tells us that there is an injective group homomorphism 
\[
\Pic(X_n)\hookrightarrow \pi_1(X_n)^{\mathrm{ab}}
\]
with dense image (see~\cite{A-T} page 59 or~\cite{CF} Chapter VII Section 5.5). Moreover, this homomorphism is normalized such that the line bundle $\O_{X_n}(-y)$ is sent to $\Frob_y\in\pi_1(X_n)$ for any $y\in X_n$. By construction this morphism commutes with the action of $\Gal(\ff_{q^n}/\ff_q)$ on each side.

Therefore we deduce that a continuous character of $\pi_1(X_n)$ over $\qqb_l$ is the same as a continuous character of $\Pic(X_n)$ over $\qqb_l$ and moreover that the isotropy groups of the corresponding characters are the same under the action of $\Gal(\ff_{q^n}/\ff_q)$. The continuous characters of $\Pic(X_n)$ (or $\pi_1(X_n)$) which have trivial isotropy group are called primitive (of degree $n$). We denote the set of primitive characters of $\Pic(X_n)$ modulo the action of the Galois group $\Gal(\ff_{q^n}/\ff_q)$ by $\tP_n$. Remark here that our previous notion of primitive character
was just for the group $\wX{n}$ but actually a character of $\Pic(X_n)$ is primitive if and only if its restriction to $\Pic^0(X_n)$ is primitive.

It is clear from the above that if we start with a primitive character of $\Pic(X_n)$ we can associate to it an irreducible $l$-adic representation of $\pi_1(X)$ of dimension $n$ and moreover this application becomes a bijection modulo the group $\Gal(\ff_{q^n}/\ff_q)$ that acts on the characters.

The Langlands correspondence (see~\cite{Laff} or~\cite{Fr1} Section 2.4 Theorem 1) asserts that for each $n\ge 1$ there exists a bijection between the unramified cusp eigenforms of rank $n$ on $X$ and the irreducible $l$-adic representations of the fundamental group $\pi_1(X)$. Moreover, this bijection satisfies the crucial rigidity property that the Hecke eigenvalues are equal to the Frobenius eigenvalues.
  
Putting all together we thus obtain:
\begin{thm}\label{T:Frobenius action} (a) There exists a bijection
\[
\left\{
\text{rank } n \text{ cusp eigenforms on }X
\right\}\simeq
\tP_n.
\]
(b) Let $f$ be a cusp eigenform on $X$ of rank $n$. Let $x\in X$ be a point of degree $d'$ and let $d:=(n,d')$. Denote by $x_1'',\dots,x_d''$ the points of $X_n$ that sit above $x$. Then the Hecke eigenvalues at the point $x$ are given by the roots of the polynomial
\[
\mathrm{char}(\Frob_x|_V)(T)=\prod_{i=1}^d \left(T^{n/d}-\rho(\O_{X_n}(-x_i''))\right)
\]
where $\rho$ is the unique (up to the Galois action) primitive character of $\Pic(X_n)$ determined by the bijection from (a).
\end{thm}

\begin{cor}\label{C:Hecke eigenvalues} Keeping the notation of (b) above, we have, for every $l=1,\dots,n$:
\begin{eqnarray*}
\He_{\O_x^{\oplus l}}(f)& = &q_x^{l(n-l)/2}e_l(z_{1,x}(f),\dots,z_{n,x}(f))f\\
& = &
\left\{
\begin{array}{ll}
0,&\text{if } \frac nd \not| \,l\\
(-1)^{l+l'}q_x^{l(n-l)/2}e_{l'}(\rho(\O_{X_n}(-x_1'')),\dots,\rho(\O_{X_n}(-x_d'')))f, & \text{if }l=l'\frac nd
\end{array}
\right.\end{eqnarray*}
where $e_k$ is the $k$-th elementary symmetric polynomial.
\end{cor}

\subsection{} In this subsection we compute the action of the Hecke operators associated to the elements $T_{(0,r)}^{\tilde\sigma}$, where $\tsigma\in\widetilde{X}$ is an arbitrary character, on the cusp eigenforms. See Section~\ref{S:def of T_r,d} for the definition of $T_{(0,r),x}$ and $T_{(0,r)}^{\tilde\sigma}$.

Recall that for any $x\in X$ we have an isomorphism $\Psi_{\kk_x}:\HH_{\Tor_x}\to\Lambda_t\mid_{t=v^{2|x|}}$ which sends $\O_x^{\oplus l}$ to $q_x^{-l(l-1)/2}e_l$ where $e_l$ is the $l$-th elementary symmetric function.

Fix $\trho\in\tP_n$ a primitive character and let $f=f_\trho$ be the corresponding cusp eigenform (see Theorem~\ref{T:Frobenius action}).

Consider the map $\Phi_{x,f}:\Lambda_t\mid_{t=v^{2|x|}}\to\C$ that sends $x_i$ to $q_x^{(n-1)/2}z_{i,x}(f)$ if $i=1,\dots,n$ and to $0$ if $i>n$.
Then we have that 
\[
\He_{\O_x^{\oplus l}}(f) = \Phi_{x,f}(\Psi_{\kk_x}(\O_x^{\oplus l}))f.
\]

Since the elements $T_{(0,r|x|),x}\in\HH_{\Tor_x}$ correspond to $[r|x|]/r p_r \in\Lambda_t\mid_{t=v^{2|x|}}$ we obtain
\[
\He_{T_{(0,r|x|),x}}(f) = \frac{[r|x|] \Phi_{x,f}(p_r)}{r} f.
\]

Using Theorem~\ref{T:Frobenius action} and the elementary formula 
\[
\sum_{l=0}^{k-1}\epsilon^{li}=\left\{\begin{array}{l}
0,\textrm{ if } k\not|i,\\
k, \textrm{ if } k|i
\end{array}\right.
\] where $\epsilon$ is a primitive $k$-th root of unity, we obtain:

\[
\Phi_{x,f}(p_r) = \left\{
\begin{array}{l}
\ds q_x^{r(n-1)/2}\frac{n}{d_x}\sum_{i=1}^{d_x}\rho(\O_{X_n}(-x'_i))^{rd_x/n},\textrm{ if } rd_x/n\in\ZZ\\
0,\textrm{ otherwise}
\end{array}
\right.
\]
where $d_x=\gcd(n,|x|)$.

We need the following easy lemma:
\begin{lemma}\label{L:Norm of x'' from N to n} Let $N\ge 1$ be an integer and $x$ be a point of $X$ such that $n$ and $|x|$ divide $N$. Let also $x'$ be a point of $X_n$ that sits above $x$ and $x''$ be a point of $X_N$ that sits above $x'$. Then we have:
\[
\Norm_n^N(\O_{X_N}(x'')) = \O_{X_n}(x')^{(N/n)/(|x|/d_x)}.
\]
\end{lemma}

Denoting by $N:=r|x|$ and using the above lemma we have:

\begin{eqnarray*}
\Phi_{x,f}(p_r) & = & \left\{
\begin{array}{l}
\ds q^{N(n-1)/2}\frac{n}{d_x}\sum_{i=1}^{d_x}\frac{d_x}{|x|}\sum_{X_N\ni x''\to x_i'}\rho(\Norm_n^N(\O_{X_N}(-x''))), \textrm{ if } rd_x/n\in\ZZ\\
0,\textrm{ otherwise}
\end{array}
\right.\\
& = & 
\left\{
\begin{array}{l}
\ds q^{N(n-1)/2}\frac{n}{|x|}\sum_{X_N\ni x''\to x}\rho(\Norm_n^N(\O_{X_N}(-x''))),\textrm{ if } rd_x/n\in\ZZ\\
0,\textrm{ otherwise}
\end{array}
\right.
\end{eqnarray*}

Observe now that for a positive integer $N$ the condition $|x|$ and $n$ divide $N$ is equivalent to $(N/|x|)/(n/d_x)\in\ZZ$.

Fix a positive integer $N$ such that $n$ divides $N$ and let $x\in X$ be such that $|x|$ also divides $N$. From the above we have:

\[
\He_{T_{(0,N),x}}(f) = \frac{[N]}{N}q^{N(n-1)/2}n\sum_{X_N\ni x''\to x}\Norm_n^N(\rho)(\O_{X_N}(-x''))
\]

Let $\tilde\sigma\in\widetilde{X}$ be a character of order $N$. We have:
\begin{eqnarray*}
\He_{T_{(0,N)}^{\tilde\sigma}}(f) & = & 
\frac{[N]n}{N}q^{N(n-1)/2}\sum_{\stackrel{x\in X}{|x|\mid N}}\sum_{X_N\ni x''\to x}\Norm_n^N(\rho)(\O_{X_N}(-x''))\cdot\\ 
& & \qquad\qquad \cdot\frac1{|x|}\sum_{X_N\ni y''\to x}\sigma(\O_{X_N}(y''))f\\
& = & \frac{[N]n}{N^2}q^{N(n-1)/2}\sum_{i=0}^{N-1}\sum_{x''\in X(\ff_{q^N})}\Norm_n^N(\rho)(\O_{X_N}(-x''))\cdot\\
& & \qquad\qquad\cdot \Fr_{X,N}^i(\sigma)(\O_{X_N}(x''))f\\
& = & \frac{[N]n}{N^2}q^{N(n-1)/2}\left|X(\ff_{q^N})\right|\sum_{i=0}^{N-1}\<\Fr_{X,N}^i(\sigma), \Norm_n^N(\rho)\>f\qquad\qquad (*)
\end{eqnarray*}
where for two characters $\chi_1,\chi_2$ of a group $A$ we denoted by \[
\<\chi_1,\chi_2\>:=\frac1{|A|}\sum_{x\in A}\chi_1(x)\chi_2(-x)
\] their scalar product as characters.

Recall that $\rho$ was a primitive character and so, by Lemma~\ref{L:norm-surjective}, we deduce that 
\[
\Fr_{X,N}^i(\Norm_n^N(\rho)) = \Norm_n^N(\rho) \textrm{ if an only if } n|i.
\]
Applying this to formula (*) we obtain the following proposition:

\begin{prop}\label{P:action of T_0N on cusp eign} Under the assumptions of this subsection we have:
\[
\He_{T_{(0,N)}^{\tilde\sigma}}(f) = \left\{
\begin{array}{ll}
0,&\textrm{ if } \tilde\sigma\neq\Norm_n^N(\tilde\rho)\\
\ds\frac{[N]}{N}q^{N(n-1)/2}|X(\ff_{q^N})|f, &\textrm{ if }\tilde\sigma = \Norm_n^N(\tilde\rho)
\end{array}
\right.
\]
\end{prop}

\begin{rem}
Observe that the constant appearing in the above formula does not depend on the cusp-eigenform $f$. We will exploit this later on to prove that all the twisted spherical Hall algebras associated to primitive characters of degree $n$ are isomorphic to a specialization of the universal spherical algebra $\E_{\sigma,\ov\sigma}^n$.
\end{rem}

\subsection{Computations of the Rankin-Selberg $L$-functions}\label{S:computations Rankin-Selberg}
Using the above results we are able now to give a proof of Proposition \ref{P:value L func}.

We begin with a lemma:

\begin{lemma}\label{L:L-function-trivial} Let $Y$ be an elliptic curve over the field $\ff_{q}$ with $y_0$ a rational point (the origin) and let $\rho$ be a character of $\Pic(Y)$ whose restriction to $\Pic^0(Y)$ is nontrivial.
Then the $L$-function 
\[
L(\rho,t):=\prod_{y\in Y}(1-\rho(\O_Y(y))t^{|y|})
\]
is identically equal to 1.
\end{lemma}

\begin{proof}
Recall that on $\ov Y$ we have the Frobenius action $\Frob_Y$ which satisfies $Y(\ff_{q^n})={\ov Y}^{\Frob_Y^n}$. By Lemma \ref{L:norm-surjective} we know that the map $\Norm_n:\Pic^0(Y_n)\to \Pic^0(Y)$ is surjective.
Therefore we have:
\begin{eqnarray*}
\log L(\rho,t) & = & -\sum_{y\in Y}\log(1-\rho(\O_Y(y))t^{|y|})\\
& = & \sum_{y\in Y}\sum_{k\ge 1}\rho(\O_Y(y))^k t^{k|y|}/k\\
& = & \sum_{n\ge 1}\frac{t^n}{n} \sum_{d\mid n}\sum_{y\in Y, |y|=d} d\rho(\O_{Y}(y))^{n/d}\\
& = & \sum_{n\ge 1}\frac{t^n}n\sum_{y'\in Y(\ff_{q^n})} \rho(\Norm_n(\O_{Y_n}(y')))\\
& = & \sum_{n\ge 1}\frac{t^n}n\rho(\O_Y(y_0))^n\sum_{y'\in Y(\ff_{q^n})}\rho(\Norm_n(\O_{Y_n}(y'-y_0)))\\
& = & \sum_{n\ge 1}\frac{t^n}n\rho(\O_Y(y_0))^n\sum_{\kL\in\Pic^0(Y_n)}\rho(\Norm_n(\kL))\\
& = & \sum_{n\ge 1}\frac{t^n}{n}\rho(\O_Y(y_0))^n \frac{\# Y_n(\ff_{q^n})}{\#Y(\ff_q)} \sum_{\kL\in\Pic^0(Y)}\rho(\kL)\\
& = & \underbrace{\left(\sum_{\kL\in\Pic^0(Y)}\rho(\kL)\right)}_{=0}\left(\sum_{n\ge 1}\frac{t^n}{n}\rho(\O_Y(y_0))^n \frac{\# Y_n(\ff_{q^n})}{\#Y(\ff_q)}
\right)\\
& = & 0
\end{eqnarray*}

where the last equality holds because $\rho|_{\Pic^0(Y)}$ is a non-trivial character.
\end{proof}

\begin{proof}(of Proposition \ref{P:value L func})
Let $f$ be a cusp eigenform of rank $n$. Denote by $\rho\in\Pic^0(X_n)$ the corresponding primitive character (see Theorem~\ref{T:Frobenius action}).
Let $x\in X$ be a closed point of degree $|x|$. Denote by $d_x:=\gcd(n,|x|)$.

The Hecke eigenvalues are given by (see Theorem~\ref{T:Frobenius action}): 
\[
z_{(i-1)n/d_x+l,x}(f)=\epsilon_x^l\rho(\O_{X_{n}}(x_i))^{d_x/n}
\]
for $i=1,\dots,d_x$ and $l=0,\dots, n/d_x-1$ where $\epsilon_x$ is a $n/d_x$-th primitive root of $1$.

Therefore the Rankin-Selberg $L$-function associated to the pair $(f,f)$ (see subsection \ref{def Rankin Selberg} for the definition) is:
\begin{eqnarray*}L(f,f,t)&:=&\prod_{x\in X}\prod_{\stackrel{i,l}{j,k}}(1-\epsilon_x^{l-k}\rho(x_i-x_j)^{d_x/n}t^{|x|})^{-1}\\
& = & \prod_{x\in X} \prod_{i,j=1}^{d_x}(1-\rho(x_i-x_j)t^{|x|n/d_x})^{-n/d_x}\\
& = & \prod_{x\in X} \prod_{i=1}^{d_x}\prod_{j=0}^{n-1}(1-\rho(x_i-\Frob_X^j (x_i))t^{|x|n/d_x})^{-1}\\
& = & \prod_{x\in X_n}\prod_{j=0}^{n-1}(1-\rho(x-\Frob_X^j(x))t^{|x|n})^{-1}\\
& = & \zeta_{X_n}(t^n)\cdot\prod_{j=1}^{n-1}\prod_{x\in X_n}(1-\rho(x-\Frob_X^j(x))t^{|x|n}))^{-1}
\end{eqnarray*}

where for a closed point $x\in X$ we denoted by $x_1,...,x_{d_x}$ the closed points of $X_n$ lying over $x$.

We can apply Lemma \ref{L:L-function-trivial} to the curve $Y=X_n$ and to the characters $\rho_j:=\rho\circ(Id-\Frob_X^j),j=1,\dots,n-1$. Note that since $\rho$ is primitive we have that each $\rho_j$ is a nontrivial character on $\Pic^0(Y)$.

It follows that $L(f,f,t)=\zeta_{X_n}(t^n)$. 
\end{proof}

\section{Some computations}\label{computations}

We grouped together in this Section several technical lemmas that we will need for the proofs of the main results.

\subsection{}
Recall that for a torsion sheaf $\tau$ we denote by $\V(\tau)$ the universal extension of $\tau$ and $\O$. By definition, this means that $\V(\tau)$ fits into an exact sequence:
\[
0\to \O\ot\Ext(\tau,\O)^*\to V(\tau)\to \tau\to 0
\]
and the class of the extension corresponds to $Id\in\End(\Ext(\tau,\O))$. 
We know from Atiyah's theorem that $\V(\tau)$ is semistable and moreover that we can obtain all the semistable sheaves of slope 1 by this process.

\begin{lemma}\label{hecke_action} Let $\tau$ be a torsion sheaf of degree $n$.
If $\V$ is a semistable sheaf that fits in an exact sequence 
\[
0\to \V\to \O^{\oplus n}\to \tau \to 0
\]
then $\V$ is isomorphic to $\V(\tau)^\vee$.
\end{lemma}
\begin{proof}

%
%
%

If $0\to\V\to\O^{\oplus n}\stackrel{f}{\to}\tau$ is an exact sequence then, by applying the functor $\Hom(\O,-)$, we have a surjective linear map $f^\#:k^n\to\Hom(\O,\tau)\simeq k^n$. If $f^\#$ had a non-zero kernel then $\O\subseteq\ker(f)=\V$ which is a contradiction since $\V$ is semistable of negative slope. It follows that $f^\#$ is an isomorphism and therefore we can conjugate $f$ by a suitable automorphism of $\O^{\oplus n}$ such that
\[
\V\simeq\ker(\ev:\O\otimes\Hom(\O,\tau)\to\tau).
\]

If we look at the sequence $0\to \O^{\oplus n}\to \V(\tau)\to \tau\to 0$ and apply $\mathcal{H}om(-,\O)$ we obtain:

\[
0\to 0\to \V(\tau)^\vee\to \O^{\oplus n}\stackrel{\psi}{\to} \mathcal{E}xt (\tau,\O)=\tau \to 0.
\]

By the above we get that $\V(\tau)^\vee\simeq\V$.
\end{proof}

As a corollary of the proof we get the following Hall numbers~:

\begin{cor}\label{C:comp Hall numbers} If $x$ is a closed point and $\lambda$ a partition such that $|x||\lambda|=n$ then:
\[P_{\O_x^{(\lambda)},\V(\O_x^{(\lambda)})^\vee}^{\O^{\oplus n}} = \ds |GL_r(q)|\cdot a_{\V(\O_x^{(\lambda)})}.\]
\end{cor}

\begin{cor}\label{C:Hecke tau applied to O^n}
  Let $\tau$ be a torsion sheaf on $X$ of degree $n$. Then 
\[
 (\He_{\tau}({\O^{\oplus n}}))^\ss = P_{\O^{\oplus n}, \tau}^{\V(\tau)} {\V(\tau)}
\]
where by the exponent $(\,\,)^\ss$ we indicate the semistable part.
\end{cor}

\begin{lemma}\label{L:dual of V(tau)} Let $\tau$ be a torsion sheaf on $X$. If we denote by $\ominus\tau$ the (point by point) opposite of $\tau$ in the group law then we have an isomorphism
\[\V(\tau)^\vee\otimes\O(2x_0)\simeq\V(\ominus\tau).\]
\end{lemma}
\begin{proof} It is enough to prove this result when $\tau=\O_x^{(r)}$ for some point $x\in X$ and some integer $r\ge 1$. Denote by $d$ the degree of $x$ and by $x_1,...,x_d$ the points of $\ov X$ that lie over $x$.

If we have now two sheaves $\kF,\kG$ on $X$ such that their pullback to $\ov X$ are isomorphic then by \cite{BS} Proposition A.1 we see that the sheaves $\kF$ and $\kG$ are actually isomorphic over $X$. So in order to prove the needed result we can pull everything back to $\ov X$ and work there.

We obviously have that $\O_x\otimes_{\ff_q}\ov\ff_q=\O_{x_1}\oplus \dots \oplus \O_{x_d}$. Therefore $\V(\O_x^{(r)})\otimes_{\ff_q}\ov\ff_q = \V(\O_{x_1}^{(r)}) \oplus \dots\oplus \V(\O_{x_d}^{(r)})$.

Hence it is enough to prove that for any closed point $y$ of $\ov X$ we have the following isomorphism of sheaves on $\ov X$: 
$\V(\O_{y}^{(r)})^\vee\otimes\O(2x_0)\simeq \V(\O_{\ominus y}^{(r)})$.

For this we will need to use the fact that $\V(\O_y^{(r)})^\vee$ is the only indecomposable sheaf that fits in an exact sequence of the form
\[
0\to\V(\O_y^{(r-1)})^\vee\to\V(\O_y^{(r)})^\vee\to\V(\O_y)^\vee\to 0 \tag{*}
\]

This in turn can be proved by observing that the functor 
\[
R(-):=\ker(\Hom(\O,-)\otimes \O\stackrel{\ev}{\to}-)
\]
is exact on the category of torsion sheaves and that $R(\tau) = \V(\tau)^\vee$.
 
It is immediate that $\V(\O_y)=\O(y)$ and that $\V(\O_y)^\vee\otimes\O(2x_0)\simeq\O(\ominus y)=\V(\O_{\ominus y})$.

By induction and using the short exact sequence ($*$) we can easily see that the vector bundle $\V(\O_x^{(r)})^\vee \otimes \O(2x_0)$ is isomorphic to $\V(\O_{\ominus y}^{(r)})$.
\end{proof}

Let $n\ge 1$ be an integer, $x$ be a closed point of degree $d$ on $X$ and $\lambda$ a partition such that $|x||\lambda|=n$. Let $\tilde\rho\in\widehat X$ be a character.

\begin{lemma}\label{L:action of O_x on T} With the above hypothesis we have
\[\<\O_x^{(\lambda)}\bullet T_{(n,n)}^{\tilde\rho}\,,\,\O^{\oplus n}(2x_0)\>=v^{n^2}\frac{[n] }{n}\frac{|x| n_{u_x}(l(\lambda)-1)}{a_{\O_x^{(\lambda)}}}\rho(\ominus x)\]
where $u_x=q_x^{-1/2}$ and $n_{u}(l):=\prod_{i=1}^l(1-u^{-2i})$.
\end{lemma}
\begin{proof}
By Corollary \ref{C:comp Hall numbers}, Lemma \ref{L:dual of V(tau)} and by the definition of the coproduct we know that 
\[
\Delta(\O(2x_0)^{\oplus n}) = v^{n^2}\frac{a_{\O_x^{(\lambda)}}a_{\V(\O_x^{(\lambda)})^\vee}}{a_{\O^{\oplus n}}}\frac{|GL_n(q)|}{a_{\O_x^{(\lambda)}}}\O_x^{(\lambda)}\ot \V(\O_{\ominus x}^{(\lambda)})+\dots
\]
where the dots mean that all the other semistable terms that appear in the coproduct contain on the first position sheaves which are orthogonal (w.r.t. Green's form) to $\O_x^{(\lambda)}$.

Therefore, by the Hopf property of Green's form, we get: \begin{eqnarray*}\label{formula_hecke_op}
\ds\<\O_x^{(\lambda)}\bullet T_{(n,n)}^\rho\,,\,\O(2x_0)^{\oplus n}\> &=& 
\<\O_x^{(\lambda)}\otimes T_{(n,n)}^\rho\,,\,\Delta(\O(2x_0)^{\oplus n})\>\\
& = & v^{n^2}\frac{[n] }{n}\frac{|x| n_{u_x}(l(\lambda)-1)}{a_{\O_x^{(\lambda)}}}\rho(\ominus x).
\end{eqnarray*}\qedhere
\end{proof}

\section{Proofs of the main results}\label{S:proof main results}
\subsection{}
In this Section we will prove the main results of this paper, namely Theorem~\ref{T:cusp forms} and Theorem~\ref{T:str of Hall}. As we said before the proof will be by induction and the argument will be quite roundabout. For this reason we summarize the main steps that we will follow:

Step $1_n$: Prove that $\left\{T_{n}^{\tilde\rho}:\tilde\rho\in\P_n\right\}$ are the cusp eigenforms corresponding to the characters $\tilde\rho\in\P_n$ (see Theorem~\ref{T:Frobenius action}).

Step $2_n$: For a primitive character $\tilde\rho\in\P_n$ and for a character $\tilde\sigma\in\widehat X$ of degree $N$ such that $n|N$ we have the formula:
\[
[T_{(0,N)}^{\tilde\sigma},T_{(n,0)}^{\tilde\rho}] =\left\{
\begin{array}{ll}
 \ds v^N\frac{[N]}{N}|X(\ff_{q^N})|T_{(n,N)}^{\tilde\rho}& \textrm{ if }\tilde\sigma = \Norm_n^N(\tilde\rho)\\
 0,& \textrm{ otherwise}
\end{array}
\right.
\]

Step $3_n$: Prove the formula for the coproduct:
\[
\Delta(T_{(n,0)}^{\tilde\rho}) = T_{(n,0)}^{\tilde\rho}\ot 1+\sum_{d\ge 0}\theta_{d}^{\tilde\rho}\ot T_{(n,-nd)}^{\tilde\rho}
\]
where the coefficients $\theta_{d}^{\tilde\rho}$ are given by equating (formally) the following two series:
\[
\sum_{d\ge 0}\theta_{d}^{\tilde\rho}s^d = \exp\left(n(v^{-1}-v)\sum_{l\ge 1}T_{(0,nl)}^{\Norm_n^{nl}(\tilde\rho)}s^l\right).
\]

For $\x\in\Z$ define the elements $\theta_\x^\trho$ by putting $\theta_{(0,d)}^\trho:=\theta_d^\trho$ for $d\ge 0$ and in general using the $\SL_2(\ZZ)$ action.

Step $4_n$: Describe the structure of $\U_X^{\tilde\rho}$. More precisely we prove that there exists a natural isomorphism of algebras $\U_X^{\trho}\simeq \E_{\sigma,\ov\sigma}^n$ where $\sigma,\ov\sigma$ are the eigenvalues of the Frobenius on $\mathrm{H}^1_{\mathrm{et}}(\ov X,\overline\Q_l)$.

Step $5_n$: Prove that $\AFc_{n+1,d}=0$ if $n+1\nmid d$.

Step $6_n$: A basis for $\AFc_{(n+1,0)}$ is given by $\left\{T_{(n+1,0)}^{\tilde\sigma}:\tilde\sigma\in\P_{n+1}\right\}$.

\vspace{0.1in}

The crucial step in the proof is the Step $1_n, n\ge 1$. The proof of this uses the Langlands correspondence and the computation of the Hecke eigenvalues for the cusp form associated to $\tilde\rho$ in terms of the character $\tilde\rho$.

\subsection{} Case $n=0$. Observe that in this case only the statement from the step 6 is not vacuous. The proof is obvious. See also Corollary \ref{C:basis primitive torsions}.

\subsection{} We can proceed to prove the steps $1_n,\dots,6_n$ for general $n\ge 1$.
\vspace{0.1in}
So let us suppose that we've proved Steps $1_m,...,6_m$ for all integers $0\le m<n$.

\textbf{\textit{Proof (of Step $1_n$)}}
From the Step $6_{n-1}$ a basis of $\AFc_{(n,0)}$ is given by $\{T_{(n,0)}^{\trho}:\trho\in\P_n\}$.

For each $\trho\in\P_n$ consider $f_{\trho}$ the associated cusp eigenform (see Theorem~\ref{T:Frobenius action}). From the Step $5_{n-1}$ we can write $f_{\trho} = \sum_d\in\ZZ f_{\trho,d}$ where $f_{\trho,d}$ is of degree $nd$. Therefore, using Lemma~\ref{L:normalization of cusp eigenforms}, we can normalize $f_{\trho}$ such that $f_{\trho}(\O^{\oplus n})=1$.

Recall that because we only consider characters on the zeroth component of the Picard group the associated cusp eigenforms verify the property $f_{\trho} = f_\trho\otimes_\O \O_X(dx_0),\forall d\in\ZZ$ and therefore $f_{\trho,d} = f_{\trho,0}\otimes_\O \O_X(dx_0)$.

From the step $6_{n-1}$ we know that a base for the space $\AFc_{(n,nd)}$ is given by the elements $T_{(n,nd)}^{\trho},\trho\in\P_n$ for any integer $d\in\ZZ$. Thus there exist constants $c_{\trho,\tsigma}$ such that $f_{\trho,d} = \sum_{\tsigma}c_{\trho,\tsigma}T_{(n,nd)}^{\tsigma},\forall d\in\ZZ$.

Let $x\in X$ be a closed point of degree $|x|$ such that $|x|$ divides $n$. We set $l=n/|x|$.

Consider the Hecke operator $\He_{\O_x^{\oplus l}}$ that acts on the cusp form $f_{\trho,1}$ and on the elements $T_{(n,n)}^{\tsigma}$.
From the eigenform property of $f_\trho$ and from the Corollary~\ref{C:Hecke eigenvalues} we have:
\[
(\He_{\O_x^{\oplus l}}(f_{\trho,1}),\O_X(2x_0)^{\oplus n}) = (-1)^{l+1}\frac1{|\GL_n(q)|}q^{n(n-l)/2}|x|\trho(\ominus x)
\]

On the other hand, by Lemma~\ref{L:action of O_x on T}, we have
\begin{eqnarray*}
(\He_{\O_x^{\oplus l}}(f_{\trho,1}),\O_X(2x_0)^{\oplus n}) & = & \sum_{\tsigma} c_{\trho,\tsigma}(\He_{\O_x^{\oplus l}}(T_{(n,n)}^\tsigma),\O_X(2x_0)^{\oplus n})\\
& = & \frac{[n]}{n}|x|\frac{n_{u_x}(l-1)}{|\GL_l(q_x)|}\sum_\tsigma c_{\trho,\tsigma}\tsigma(\ominus x)\\
& = & \frac{[n]}n |x| (-1)^{l-1}\frac{q^{-n(l-1)/2}}{q^n-1} \sum_\tsigma  c_{\trho,\tsigma}\tsigma(\ominus x)
\end{eqnarray*}

Equating the two expressions gives:
\begin{eqnarray*}
\trho(\ominus x) & = & \frac{[n]}{n}\frac{q^{-n(l-1)/2}q^{-n(n-l)/2}}{q^n-1}|\GL_n(q)|\sum_\tsigma c_{\trho,\tsigma}\tsigma(\ominus x)\\
& = & \frac{[n]}n q^{-n(n-1)/2}\prod_{i=1}^{n-1}(q^n-q^i)\sum_\tsigma c_{\trho,\tsigma}\tsigma(\ominus x)\\
& = & (-1)^{n-1}\frac{[n]}n n_v(n-1) \sum_\tsigma c_{\trho,\tsigma}\tsigma(\ominus x)
\end{eqnarray*}

Denote by $\alpha:=(-1)^{n-1}\frac{[n]}n n_v(n-1)$, by $C$ the matrix $(c_{\trho,\tsigma})_{\trho,\tsigma\in\P_n}$ and by $A$ the matrix $(\trho(x))_{\trho,x}$ where $\trho$ runs over the primitive characters in $\P_n$ and $x$ runs over the closed points of $X$ such that $|x|$ divides $n$.

What we have proved above can be rewritten in the following matrix form:
\[
CA = \alpha^{-1} A.
\]
Let us make the following simple but crucial observation: the matrix $A$ has maximal rank. Indeed, this is equivalent to the fact that the lines of the matrix, i.e. $(\trho(x))_{x}, \trho\in\P_n$, are linearly independent. But this is easily seen to be true from the definition of $\trho(x)$ and from the fact that $\trho$ are primitive.

As $C$ is a square matrix we have therefore that $C=\alpha^{-1} \mathrm{I}$, where $\mathrm{I}$ is the identity matrix. This implies that $f_{\trho,0} = \alpha^{-1}T_{(r,0)}^\trho$, or in other words that $T_r^\trho$ is, up to multiplication by a nonzero constant, the cusp eigenform associated to the character $\trho$. 

\subsection{\textit{Proof (of Step $2_n$)}}

This follows from the fact that $T_n^{\trho}$ is the cusp eigenform associated to $\trho$ combined with Proposition~\ref{P:action of T_0N on cusp eign}.

\subsection{\textit{Proof (of Step $3_n$)}} Fix a character $\trho\in\P_n$.
From Lemma~\ref{L:Delta(cusp)=cusp} and step $1_n$ we can write:
\[
\Delta(T_{(n,0)}^\trho) = T_{(n,0)}^\trho\ot 1+\sum_{d\ge 0}\sum_{\tsigma\in\P_n} \theta_{d}^\tsigma \ot T_{(n,-nd)}^\tsigma
\]
where $\theta_{d}^\tsigma\in\HH_X[0,nd]$.

Let us first prove that $\theta_d^\tsigma=0$ for $\tsigma\neq\trho$. Using the Hopf property of the Green form we have:
\begin{eqnarray*}
\|\theta_d^\tsigma\|^2\|T_{(n,-nd)}^\tsigma\|^2 & = & (\Delta(T_{(n,0)}^\trho),\theta_d^\tsigma\ot T_{(n,-nd)}^\tsigma)\\
& = & (T_{(n,0)}^\trho,\theta_d^\tsigma\cdot T_{(n,-nd)}^\tsigma)\\
& = & \alpha (T_{(n,0)}^\trho,T_{(n,0)}^\tsigma) \\
& = & 0
\end{eqnarray*}
for some $\alpha\in\C$ and where the last equality follows from  Lemma~\ref{L:scalar prod T_rho}.
Since $T_{(n,-nd)}^\tsigma$ is non-zero we deduce that $\theta_d^\tsigma=0$ for all characters $\tsigma\neq\trho$ and all integers $d>0$.

Let us consider the generating series 
\[
T_n^\trho(s):=\sum_{d\in\ZZ} T_{(n,nd)}^\trho s^d
\textrm{ and}
\]
\[
\theta^\trho(s) = \sum_{d\ge 0}\theta_d^\trho s^d.
\]
We easily deduce the following formula:
\[
\Delta(T_n^\trho(s)) = T_n^\trho(s)\ot 1+\theta^\trho(s)\ot T_n^\trho(s).
\]
From the associativity of $\Delta$ we get that $\Delta(\theta^\trho(s)) = \theta^\trho(s)\ot\theta^\trho(s)$.

Consider the generating series $a^\trho(s) := \log(\theta^\trho(s))$. By the above formula for the coproduct of $\theta^\trho(s)$ we have that $\Delta(a^\trho(s)) = a^\trho(s)\ot 1+1\ot a^\trho(s)$. In other words, if we write $a^\trho(s) = \sum_{d\ge 1}a^\trho_d s^d$, the elements $a^\trho_d\in \HH_X[0,nd]$ are primitive for the coalgebra structure.

Using Corollary~\ref{C:basis primitive torsions} we can write every element $a^\trho_d$ as a linear combination of $T_{(0,nd)}^\txi$, say
\[
 a^\trho_d = \sum_{\txi\in\wXF{nd}}\alpha_d^\txi T_{(0,nd)}^\txi.
\]
We would like to prove that only the element $T_{(0,nd)}^{\Norm_n^{nd}(\trho)}$ appears in $a^\trho_d$ with a nonzero coefficient.

Suppose that it is not the case. Choose the minimal $d>0$ such that there exists a character $\tsigma\in\wXF{nd},\tsigma\neq\Norm_n^{nd}(\trho)$ with $\alpha_d^\tsigma\neq 0$.

Then we have $(\theta_d^\trho,T_{(0,nd)}^\tsigma) = (a_d^\trho,T_{(0,nd)}^\tsigma)\neq 0$ where the  inequality follows from the minimality of $d$ and the orthogonality of the $T_{(0,nd')}^\txi$ for different $\txi\in\widetilde X,d'\ge 0$ (see Lemma~\ref{L:scalar prod T_rho}).

Hence we have the following equalities:
\begin{eqnarray*}
0\neq (\theta_d^\trho,T_{(0,nd)}^\tsigma)\|T_{(n,-nd)}^\trho\|^2 & = & (\Delta(T_{(n,0)}^\trho),T_{(0,nd)}^\tsigma\ot T_{(n,-nd)}^\trho)\\
& = & (T_{(n,0)}^\trho,T_{(0,nd)}^\tsigma\cdot T_{(n,-nd)}^\trho)\\
& = & 0
\end{eqnarray*}
the last equality follows by Proposition~\ref{P:action of T_0N on cusp eign}. Contradiction.

Therefore $a_d^\trho = \alpha_d^\trho T_{(0,nd)}^{\Norm_n^{nd}(\trho)}$. We now need to identify the coefficient $\alpha_d^\trho$.

To this end fix $d$ and let $x\in X$ be a closed point of degree $nd$. The torsion sheaf $[\O_x]$ appears in $\theta_d^\trho$ only through $a_d^\trho$ because any monomial in the elements $a_{d'}^\trho, d'<d$ is orthogonal to $[\O_x]$. We then have the following equality:
\[
(\theta_d^\trho,[\O_x]) = (a_d^\trho,[\O_x]) = \frac{\alpha_d^\trho}{q_x-1}\frac{[nd]}{nd}|x|\Norm_n^{nd}(\trho)(x).\tag{*}
\]

On the other hand we have
\begin{eqnarray*}
(\theta_d^\trho,[\O_x])\|T_{(n,-nd)}^\trho\|^2 & = &
(\Delta(T_{(n,0)}^\trho),[\O_x]\ot T_{(n,-nd)}^\trho)\\
 & = & (T_{(n,0)}^\trho, [\O_x]\cdot T_{(n,-nd)}^\trho) \\
& \overset{Cor.~\ref{C:Hecke eigenvalues}}{=} & v^{n^2d}q_x^{(n-1)/2}n\ov{\trho(\ominus x)} \,\|T_{(n,0)}^\trho\|^2\\
& = & nv^{nd}\trho(x)\, \|T_{(n,0)}^\trho\|^2
\label{eq:}
\end{eqnarray*}

Putting this together with equation (*) and using Lemma~\ref{L:rho(x)=Norm(rho)(x)} we obtain
\[
\alpha_d^\trho = n(v^{-1}-v).
\]

All in all we have that the elements $\theta_d^\trho$ are given by equating the formal coefficients of the following two series:
\[
1+\sum_{d\ge 1}\theta_d^\trho s^d = \exp\left(n(v^{-1}-v)\sum_{d\ge 1}T_{(0,nd)}^{\Norm_n^{nd}(\trho)} s^d\right)
\]
and this is exactly what we wanted to prove.\qed

\subsection{\textit{Proof (of Step $4_n$)}} We want to prove that there is a natural isomorphism of algebras
\[
  \UU_{\sigma,\ov\sigma}^n \simeq \U_X^\trho 
\]
given by 
\[
 t_\x \mapsto T_{n\x}^{\Norm_n^{n\delta(\x)}(\trho)}.
\]

The proof is identical to the proof of Theorem 5.4 in~\cite{BS} once we show that the algebra $\U_X^{\trho,+}$ has a PBW-type decomposition. More precisely we want to prove the following proposition:

\begin{prop}\label{P:PBW for twisted sph} The multiplication maps:
\[
m:\underset{\mu\in\Q\cup\{\infty\}}{\overrightarrow{\bigotimes}'} \U_X^{\trho,\pm,(\nu)}\to \U_X^{\trho,\pm}
\]
\[
m:\underset{\mu\in\Q\cup\{\infty\}}{\overrightarrow{\bigotimes}'} \U_X^{\trho,+,(\nu)}\ot \underset{\mu\in\Q\cup\{\infty\}}{\overrightarrow{\bigotimes}'} \U_X^{\trho,-,(\nu)}\to\U_X^\trho
\] 
induce isomorphisms of vector spaces.
\end{prop}
\begin{proof}
For the proof we will use the following lemma:
\begin{lemma}\label{L:generators for the twisted spherical} Let $\trho\in\P_n$ be a primitive character. Then the algebra $\U_X^{\trho}$ is generated by the following set of elements 
\[
\left\{T_{(\pm n,0)}^\trho, T_{(0,nd)}^{\Norm_n^{n|d|}(\trho)} \mid d\in\ZZ\right\}.
\]
Similarly, the algebra $\U_X^{\trho,+}$ is generated by
\[
 \left\{T_{(n,-nd)}^\trho, T_{(0,nd)}^{\Norm_n^{nd}(\trho)} \mid d\ge 0\right\}.
\]

\end{lemma}
\begin{proof}
We only prove the first statement the second one being completely analogous.

Denote for the moment by $A$ the subalgebra of $\U_X^\trho$ generated by the set from the statement. The goal is to prove that $A=\U_X^\trho$. As the algebra $\U_X^\trho$ is generated by the elements $T_{n\x}^{\Norm_n^{n\delta(\x)}(\trho)},\x\in\Z$ it is enough to prove that these generators belong to $A$.

First observe that all the elements of the form $T_{(\pm n,nd)}^\trho,d\in\ZZ$ are in $A$. Indeed, this follows from the step 2 and from Drinfel'd's relations in the double together with the coproduct formula of step 3.

We now argue by induction. Fix an integer $k$ and suppose that we have proved that $T_{(rn,dn)}^{\Norm_n^{n\delta(r,d)}(\trho)}\in A$ for all $(r,d)\in\Z$ with $|r|<|k|$.

Let us consider the element $T_{\z}^{\Norm_n^{n\delta(\z)}(\trho)}$ for some point $\z=(k,l)\in \Z$. Take $\x=(r,d)\in\Z$ to be the closest point to the segment $[(0,0),(k,l)]$ in $\Z$ which is not proportional to $(k,l)$ and such that the triangle $\Delta_{\x,\z-\x}$ is positively oriented.

From the choice of $\x$ we have that $\delta(\x)=1$ and $\delta(\z-\x)=1$. We can find a matrix $\gamma$ in $\SL_2(\ZZ)$ such that $\gamma\cdot\x=(1,e)$ and $\gamma\cdot(\z-\x) = (-1,f)$ where $e,f$ are some integers verifying $e+f=\delta(\z)>0$\footnote{here we use the fact that triangle $\Delta_{\x,\z-\x}$ is positively oriented}. We have $\gamma\cdot \z=(0,e+f)$.

Now $\gamma$ induces an automorphism of $\U_X^\rho$ (which, a priori, might not leave $A$ stable!) that sends $T_{n\x}^{\trho}$ to $T_{(n,en)}^\trho$, $T_{n(\z-\x)}^\trho$ to $T_{(-n,fn)}^\trho$ and $T_{n\z}^{\Norm_n^{n\delta(\z)}(\trho)}$ to $T_{(0,n\delta(\z))}^{\Norm_n^{n\delta(\z)}}$.

From the Drinfel'd's relations in the double and from the computation of the coproduct of $T_{(n,0)}^\trho$ we deduce that the commutator of $T_{(n,en)}^\trho$ and $T_{(-n,fn)}^\trho$ is a linear combination of monomials in the elements $T_{(0,tn)}^{\Norm_n^{tn}(\trho)}$ where $0<t\le e+f$, and moreover that $T_{(0,n(e+f))}^{\Norm_n^{n(e+f)}(\trho)}$ appears with non-zero coefficient. 

Now pulling everything back with the automorphism induced by $\gamma^{-1}$ we obtain that the commutator of $T_\x^{\trho}$ and $T_{\z-\x}^\trho$ is a linear combination of monomials in the elements $T_{nt\z_0}^{\Norm_n^{nt}(\trho)}$ for $0<t\le \delta(\z)$, where $\z_0=\z/\delta(\z)$, and moreover that $T_{n\z}^{\Norm_n^{n\delta(\z)}(\trho)}$ appears with non-zero coefficient. Using the induction hypothesis we deduce therefore that $T_{n\z}^{\Norm_n^{n\delta(\z)}(\trho)}$ belongs to $A$ and this finishes the proof of the lemma.
\end{proof}
As an immediate corollary of this lemma combined with the formula in step 3 we obtain that:
\begin{cor}\label{C:U_Xrho is a coalgebra}
The algebras $\U_X^{\trho,+},\U_X^\trho$ are stable by the coproduct.
\end{cor}
\begin{cor}(of the Lemma, cf. \cite{BS} Corollary 4.7)\label{C:U_X is the drinfeld double}
The algebra $\U_X^\trho$ is isomorphic to the Drinfeld double of $\U_X^{\trho,+}$ and the multiplication map $m:\U_X^{\trho,+}\otimes\U_X^{\trho,-}\to\U_X^\trho$ induces an isomorphism of vector spaces.
\end{cor}

Let's get back to the proof of the Proposition~\ref{P:PBW for twisted sph}.
We first want to prove that the multiplication map:
\[
m:\HH_X^\vvec\ot \U_X^{\trho,+,(\infty)}\to \HH_X
\]
contains $\U_X^{\trho,+}$ in its image. The proof of this statement is identical to the proof of \cite{BS} Lemma 4.9 and so we skip it. We conclude as in \cite{BS} Theorem 4.8 using the $\SL_2(\ZZ)$ symmetry.
\end{proof}

Recall from~\cite{BS} Section 5 the definition of convex paths in $\Z$. We denote by $\Conv^+$ the set of positive paths. For a primitive character $\trho\in\P_n$ we define for every convex path $\p=(\x_1,\dots,\x_r)\in \Conv$ the element:
\[
 T_{n\p}^{\trho}:=T_{n\x_1}^{\Norm_n^{n\delta(\x_1)}(\trho)}\cdot\ldots\cdot T_{n\x_r}^{\Norm_n^{n\delta(\x_r)}(\trho)}.
\]
The following is an immediate corollary of the Proposition~\ref{P:PBW for twisted sph}:
\begin{cor}\label{C:PBW basis convex paths for spherical}  A $\C$-basis of $\U_X^{\trho}$, where $\trho\in\P_n$, is given by:
\[
\left\{T_{n\p}^\trho, \p\in \Conv^+\right\}.
\]
\end{cor}

\subsection{\textit{Proof (of Step $5_n$)}} The goal is to prove that $\AFc_{n+1,d}=0$ if $n+1\nmid d$.

We know from Proposition~\ref{P:cusp perp} that $\AFc_{n+1,d} = \HH_X^\vvec[n+1,d]\cap(\HH_X^{\le n})^\perp$.

We will prove that $\HH_X^\vvec[n+1,d]\subset \HH_X^{\le n}$ if $n+1\nmid d$. It is enough to treat only the case $0<d<n+1$ because $\HH_X^{\le n}$ is stable by tensor product with $\O_X(x_0)$.

Denote by $\nu = d/(n+1)\in\Q$. As the Harder-Narasimhan filtration splits it is clear that all the non-semistable sheaves in $\HH^\vvec[n+1,d]$ are already in $\HH_X^{\le n}$. Therefore we only need to deal with the semistable part $\HH_X^\vvec[n+1,d]\cap\HH_X^{(\nu)}$

Recall that the Hall algebra $\HH_X^{(\nu)}$ of semistable sheaves of slope $\nu$ is isomorphic to $\HH_X^\tor$ (by Atiyah's theorem) and a set of generators is given by $T_{l\x_0,x},x\in X, l\ge 1, |x|\mid l$, where $\x_0 = (r',d'), r'\ge 1$ such that $\mu(\x_0) = \nu$ and $\delta(\x_0)=1$.

Therefore a basis for the space $\HH_X[n+1,d]\cap\HH_X^{(\nu)}$ is given by 
\[
 \left\{T_{l_1\x_0,x_1}\cdots T_{l_k\x_0,x_k}\mid k\ge 1, l_i\ge 1, r'\sum_{i=1}^k l_i=n+1, x_i\in X, |x_i|\mid l_i \right\}
\]
If $l\x_0\neq (n+1,d), lr'\le n$ then obviously the elements $T_{l\x_0,x},x\in X$ are already in $\HH_X^{\le n}$ hence so is any product of them. We are left then to prove that the elements $T_{(n+1,d),x},x\in X$ are also in $\HH_X^{\le n}$. For this, observe that any such element can be written as a linear combination of elements of the form $T_{(n+1,d)}^\tsigma,\tsigma\in\wXF{k}$ where we denoted by $k:=\gcd(n+1,d)$. 

So let $\tsigma\in\wXF{k}$ be a character. By Corollary~\ref{C:prim char} there exists a primitive character $\trho\in\wXF{k'}$, where $k=k'e, e\ge 1$, such that $\tsigma = \Norm_{k'}^k(\trho)$. 
By the definition of the algebra $\U_X^{\trho,+}$ we have that $T_{(n+1,d)}^\tsigma = T_{k'(e\x_0)}^{\Norm_{k'}^k(\trho)}\in\U_X^{\trho,+}$, and from  Lemma~\ref{L:generators for the twisted spherical} we have that $\U_X^{\trho,+}\subset\HH_X^{\le k'}$. 
Now, as $n+1\nmid d$, we have that $k'\le k<n+1$ and hence $\U_X^{\trho,+}\subset \HH_X^{\le n}$. We conclude that $T_{(n+1,d)}^\tsigma\in\HH_X^{\le n}$ for any $\tsigma\in\wXF{k}$ and therefore that $\AFc_{(n+1,d)}=0$.

\subsection{\textit{Proof (of Step $6_n$)}} We have arrived at the final step of the proof. We want to prove the following statement:
\begin{thm}\label{T:base AFc(n+1,0)}
 A base for $\AFc_{n+1,0}$ is given by $\{T_{(n+1,0)}^\trho,\trho\in\P_{n+1}\}$.
\end{thm}
\begin{proof}
As in the proof of the previous step we write the space $\AFc_{n+1,0} = \HH_X[n+1,0]\cap (\HH_X^{\le n})^\perp$. Moreover we can restrict our attention to the semistable part 
\[
\HH_X^\ss[n+1,0]=\HH_X[n+1,0]\cap\HH_X^{(0)}
\] of $\HH_X[n+1,0]$ since the cuspidals are supported on semistable vector bundles. 

Now from Atiyah's theorem we know that $\HH_X^{(0)}$ is isomorphic to $\HH_X^{(\infty)}$ and hence that it is a polynomial algebra in the elements $T_{(l,0),x},x\in X, |x|\mid l$. Therefore we can write any element of $\HH_X^{\ss}[n+1,0]$ as a linear combination of monomials of the form $T_{(l_1,0)}^{\tsigma_1}\dots T_{(l_r,0)}^{\tsigma_r}$ with $l_i\ge 1, \sum_i l_i=n+1$ and $\tsigma_i = \Norm_{n_i}^{l_i}(\trho_i)$ for some $\trho_i\in\P_{n_i}$. 

From the above description it is clear that only the linear combinations of $T_{(n+1,0)}^{\tsigma}$, where $\tsigma = \Norm_k^n(\trho)$ for some $\trho\in\P_k$, could be cuspidals since all the other monomials are already in $\HH_X^{\le n}$.

If $\tsigma = \Norm_{k}^{n+1}(\trho)$ for some $\trho\in\P_k, k\le n$ then, from Lemma~\ref{L:generators for the twisted spherical} it follows that $T_{(n+1,0)}^\tsigma\in U_X^{\trho}\subset \HH_X^{\le n}$.

We are left to check that indeed the elements $T_{(n+1,0)}^\trho,\trho\in\P_{n+1}$ are cuspidal. For this it is enough to check that they are orthogonal to $(\HH_X^{\le n})^\ss[n+1,0]$. From the above discussion the elements of this last vector space are linear combination of monomials of the form
$T_{(l_1,0)}^{\tsigma_1}\dots T_{(l_r,0)}^{\tsigma_r}$ with $l_i\ge 1, \sum_i l_i=n+1$ and $\tsigma_i = \Norm_{n_i}^{l_i}(\trho_i)$ for some $\trho_i\in\P_{n_i}$ such that  $n_i\le n$. 

Let $\trho\in\P_{n+1}$ and consider a monomial $T_{(l_1,0)}^{\tsigma_1}\dots T_{(l_r,0)}^{\tsigma_r}$ as above. We have by the $\SL_2(\ZZ)$ symmetry and the Hopf property of the Green form:
\begin{eqnarray*}
 (T_{(n+1,0)}^\trho,T_{(l_1,0)}^{\tsigma_1}\dots T_{(l_r,0)}^{\tsigma_r})&
= & (T_{(0,n+1)}^\trho,T_{(0,l_1)}^{\tsigma_1}\dots T_{(0,l_r)}^{\tsigma_r})\\
& = & (\Delta^{(r)}(T_{(0,n+1)}^\trho),T_{(0,l_1)}^{\tsigma_1}\ot\dots \ot T_{(0,l_r)}^{\tsigma_r})\\
& = & \left\{
  \begin{array}{ll}
   0, &\textrm{ if } r>1 \textrm{ because } T_{(0,n+1)}^\trho \textrm{ is primitive}\\
   0, &\textrm{ if } r=1 \textrm{ because } \tsigma_1\neq \trho
  \end{array}
\right.
\end{eqnarray*}
This proves that $T_{(n+1,0)}^\trho$, where $\trho\in\P_{n+1}$, are cuspidal and hence the theorem follows.
\end{proof}

All the steps outlined at the beginning of the Section are now proved. 

\subsection{End of the proofs} We are ready now to finish the proofs of the main results Theorem~\ref{T:cusp forms}, Theorem~\ref{T:str of Hall} and Theorem~\ref{T:higher Eis rels}.

\textit{Proof of Theorem~\ref{T:cusp forms}}
Steps 1,5,6 provide the proofs of the statements of the theorem.\qed

For the proof of the theorem~\ref{T:str of Hall} we need another few lemmas:

\begin{lemma}\label{L:sphericals commute} The subalgebras $\U_X^{\tsigma}$ for $\tsigma\in\P_k$, $k\le n$, centralize each other. Moreover, there exists an isomorphism of algebras:
\[
 \bigotimes_{\substack{\tsigma\in\P_k \\ 1\le k\le n}}\U_X^{\tsigma,+}\to \HH_X^{\le n}.
\]
\end{lemma}
\begin{proof}
Let $\trho_1\in\P_{n_1},\trho_2\in\P_{n_2}$ be two \textit{different} primitive characters with $n_1,n_2\le n$. We will prove that the algebras $\U_X^{\trho_1}$ and $\U_X^{\trho_2}$ commute and this will imply the commutativity of their positive parts.

Using Lemma~\ref{L:generators for the twisted spherical} it is enough to prove that the following two sets of generators commute: 
\[
\left\{T_{(\pm n_1,0)}^{\trho_1},T_{(0,n_1d)}^{\Norm_{n_1}^{n_1|d|}(\trho_1)} \mid d\in\ZZ\right\}
\]
\[
\left\{T_{(\pm n_2,0)}^{\trho_2},T_{(0,n_2d)}^{\Norm_{n_2}^{n_2|d|}(\trho_2)} \mid d\in\ZZ\right\}
\]
What is immediate is that the torsion (resp. vector bundle) generators from one set commute with the torsion (resp. vector bundle) generators from the other set since the algebra $\HH_X^{(\infty)}=\HH_X^\tor$ (resp. $\HH_X^{(0)}$) is commutative.

The commutation of $T_{(n_1,0)}^{\trho_1}$ and $T_{(0,n_2d)}^{\Norm_{n_2}^{n_2d}(\trho_2)}$ for $d\in\ZZ_{\ge 0}$ follows from Proposition~\ref{P:action of T_0N on cusp eign} (cf. Step $2_n$). The other commutations follow from this and Drinfel'd's relations in the double.

It is clear that the subalgebras $\U_X^{\tsigma,+},\tsigma\in\P_k,k\le n$ generate $\HH_X^{\le n}$ (cf. Proposition~\ref{P:cusps generate H_X}). So we have a surjective homomorphism of algebras:
\[
 m:\bigotimes_{\substack{\tsigma\in\P_k \\ 1\le k\le n}}\U_X^{\tsigma,+}\to \HH_X^{\le n}
\]
given by the multiplication. We need to prove that the morphism $m$ is injective. For this we will use Proposition~\ref{P:PBW for twisted sph}. We have an isomorphism of vector spaces induced by the multiplication:
\[
\bigotimes_{\substack{\tsigma\in\P_k \\ 1\le k\le n}}\underset{\mu\in\Q\cup\{\infty\}}{\overrightarrow{\bigotimes}'} \U_X^{\tsigma,+,(\nu)}
\to
\bigotimes_{\substack{\tsigma\in\P_k \\ 1\le k\le n}}\U_X^{\tsigma,+}
\]
so, by using the fact that the subalgebras $\U_X^{\tsigma,+}$ commute one with each other, we get:
\[
 \underset{\mu\in\Q\cup\{\infty\}}{\overrightarrow{\bigotimes}'}\left(\bigotimes_{\substack{\tsigma\in\P_k \\ 1\le k\le n}} \U_X^{\tsigma,+,(\nu)}\right)
 \to
 \underset{\mu\in\Q\cup\{\infty\}}{\overrightarrow{\bigotimes}'} \HH_X^{\le n, (\nu)}
 \hookrightarrow
 \underset{\mu\in\Q\cup\{\infty\}}{\overrightarrow{\bigotimes}'} \HH_X^{(\nu)}
\]
So the injectivity of the morphism $m$ is implied by the injectivity of the multiplication maps:
\[
 m^{(\nu)}:\bigotimes_{\substack{\tsigma\in\P_k \\ 1\le k\le n}} \U_X^{\tsigma,+,(\nu)}\to \HH_X^{(\nu)}
\]
for every $\nu\in\Q\cup\{\infty\}$.

Using the $\SL_2(\ZZ)$ invariance it is enough to prove that $m^{(\infty)}$ in injective.
Observe here that each factor $\U_X^{\tsigma,+,(\infty)}$ is in fact a commutative bialgebra and moreover that they mutually commute. Therefore the morphism $m^{(\infty)}$ is a homomorphism of algebras. 

Observe also that each of the algebras $\U_X^{\tsigma,+,(\infty)}$, where $\sigma\in\P_k$, has a set of generators (as an algebra) consisting of primitive elements:
\[
 \Sigma_\tsigma:=\left\{T_{(0,kd)}^{\Norm_k^{kd}(\tsigma)}:d\ge 1\right\}
\]
Moreover the set 
\[
\Sigma = \coprod_{\tsigma\in\P_k,k\le n} \Sigma_\tsigma
\]
is linearly independent in $\HH_X^{(\infty)}$ because the characters, being different, are linearly independent.

Now let us recall the following result: in a commutative and cocommutative bialgebra $(C,\cdot,\Delta)$ over a field of characteristic 0, a set of primitive elements is algebraically independent if and only if it is linearly independent. 

It is easy to see that the injectivity of $m^{(\infty)}$ is equivalent to the algebraic independence of $\Sigma$ inside $\HH_X^{(\infty)}$. Since $\Sigma$ consists of primitive elements, from the above result we get that $m^{(\infty)}$ is injective if and only if $\Sigma$ is a linearly independent set in $\HH_X^{(\infty)}$. But this we've already seen to be true, hence the result follows.
\end{proof}

\begin{cor}\label{C:str of Hall prod of sphericals} The multiplication maps 
\[
\sideset{}{^\prime}{\bigotimes}_{\trho\in\P} \U_X^{\trho,+} \to \HH_X
\]
\[
\sideset{}{^\prime}{\bigotimes}_{\trho\in\P} \U_X^{\trho} \to \DH_X
\]
are isomorphism of algebras.
\end{cor}
\begin{proof}
The first isomorphism follows from the Lemma~\ref{L:sphericals commute} and the second is a corollary of the lemma plus the Corollary~\ref{C:U_X is the drinfeld double}.
\end{proof}

%

We are now ready to give the proof of \textit{Theorem~\ref{T:str of Hall}}:

The statement (1) follows from the Step 4. Statements (2) and (3) follow from Lemma~\ref{L:sphericals commute} and Corollary~\ref{C:str of Hall prod of sphericals}.
\qed

\appendix
\section{}
\begin{lemma}\label{L:proof prim char} A character $\rho\in\wX{n}$ is primitive if and only if there does not exist a character $\chi\in\wX{d}, d<n, d|n$,  such that $\rho=\Norm_d^n(\chi)$.
\end{lemma}

\begin{proof} One direction is obvious: namely, if $\rho$ is primitive then it cannot be equal to $\Norm_d^n(\chi)$ for any $\chi\in\wX{d}$ since the later has the orbit under Frobenius of cardinal $d<n$ which contradicts the primitivity of $\rho$.

Conversely, by an extension of scalars it is enough to prove that if $\rho\in\wX{n}$ is fixed by the Frobenius $\Fr_{X,n}$ then there exists a character $\chi\in\wX{}$ such that $\rho = \Norm_1^n(\chi)$. 

In order to simplify the notation and to avoid confusion we will denote (solely in this proof) by $N:=\Norm_1^n:\Pic^0(X_n)\to\Pic^0(X)$ and by $\hat{N}:=\Norm_1^n:\wX{}\to\wX{n}$.
The problem can be restated as follows: the image of the map $\hat N$ is exactly $\wX{n}^{\Fr_{X,n}}$. It is easy to see by the definition that the image of $\hat N$ is indeed contained in $\wX{n}^{\Fr_{X,n}}$. It is therefore enough to prove that $\mathrm{Im}(\hat N)$ and $\wX{n}^{\Fr_{X,n}}$ have the same cardinal.
We have that 
\[
\left|\wX{n}^{\Fr_{X,n}}\right| = \left|\Pic^0(X_n)\right|/\left|\mathrm{Im}(\Fr_{X,n}^*-\mathrm{Id})\right| = |\Pic^0(X)|.
\]

Let $\xi\in\ker\hat N$. Now Lemma~\ref{L:norm-surjective} (applied for $Y=\Pic^0(X)$) says that $N$ is surjective. Therefore we have $\hat{N}(\xi)=1$ is equivalent to $\xi(\Pic^0(X))=1$ which means that $\xi$ is the trivial character. We deduce that $\hat{N}$ is injective and hence $\mathrm{Im}(\hat N) = |\Pic^0(X)|$ which finishes the proof.
\end{proof}

\begin{lemma}\label{L:rho(x)=Norm(rho)(x)}
Let $x\in X$ be a point of degree $N$ and let $n|N$ be a positive integer. Let also $\trho\in\wXF{n}$ be a character. Then:
\[
\trho(x) = \Norm_n^N(\trho)(x).
\]
\end{lemma}
\begin{proof}
Let $x_1,\dots,x_n\in X_n$ be the points that sit over $x$ and let $x_i^j, j=1,\dots,N/n$ be the points of $X_N$ that sit over $x_i$ for $i=1,\dots,n$. By definition we have:
\begin{eqnarray*}
\Norm_n^N(\trho)(x) & = & \frac1N \sum_{i,j}\Norm_n^N(\rho)(\O_{X_N}(x_i^j))\\
& = & \frac1N\sum_{i,j}\rho\left(\otimes_{k=0}^{N/n-1} (\Frob_{X,N}^{*})^{kn} \O_{X_N}(x_i^j) \right)\\
& = &\frac1N\sum_{i=1}^n\sum_{j=1}^{N/n}\rho\left(\otimes_{k=0}^{N/n-1}\O_{X_N}(x_i^{j+k})\right)\\
& = & \frac1N\sum_{i=1}^n\sum_{j=1}^{N/n}\rho(\O_{X_n}(x_i))\\
& = & \frac1N\frac Nn\sum_{i=1}^n\rho(\O_{X_n}(x_i))\\
& = & \trho(x)
\end{eqnarray*}
\end{proof}

\begin{lemma}\label{L:scalar prod T_rho}
Let $\trho\in\P_n$ and $\tsigma\in\wXF{n}$ be two characters. Then we have:
\[
(T_{(n,nd)}^\trho,T_{(n,nd)}^\tsigma) = \left\{
\begin{array}{ll}
	0 & ,\,\mathrm{ if }\,\,\tsigma\neq\trho\\
	\ds\frac{v^n[n]}{(v^{-1}-v)n^2}|X(\ff_{q^n})| & ,\,\mathrm{ if }\,\,\tsigma=\trho
\end{array}
\right.
\]
\end{lemma}
\begin{proof}
By the $SL_2(\ZZ)$ invariance the computation of the scalar product reduces to the computation of the scalar product of $T_{(0,n)}^\trho$ and $T_{(0,n)}^\tsigma$ in $H_X^\tor$.

For a point $x\in X$ of degree $d$ such that $d\mid n$ we have from Proposition~\ref{P:classical Hall & Macd} and from the definition of $T_{(0,n),x},x\in X$:
\[
(T_{(0,n),x},T_{(0,n),x}) = \frac{v^n[n]d}{(v^{-1}-v)n}.
\]
Now taking a sum over all points $x$ we have:
\begin{eqnarray*}
(T_{(0,n)}^\trho,T_{(0,n)}^\tsigma) & = &\frac{v^n[n]}{(v^{-1}-v)n} \sum_{d|n}d\sum_{x\in X, |x|=d}\trho(x)\ov{\tsigma(x)}
\end{eqnarray*}

Let us compute the sum which appears on the right hand side. For this we will use the following notation: for a point $x\in X$ and $x'\in X_n$ we write $x'\to x$ if $x'$ sits above $x$.
\begin{eqnarray*}
\sum_{d|n}d\sum_{\substack{x\in X\\ |x|=d}}\trho(x)\ov{\tsigma(x)} & = &
\frac1{n^2} \sum_{d|n}\sum_{\substack{x\in X\\ |x|=d}} \sum_{\substack{x'\to x\\x'\in X_n}} \sum_{i=0}^{n-1}\Fr_{X,n}^i(\rho)(\O_{X_n}(x'))\cdot\\
& & \qquad\cdot \sum_{j=0}^{n-1}\ov{\Fr_{X,n}^j(\sigma)(\O_{X_n}(x'))}\\
& = & \frac1{n^2}\sum_{x'\in X(\ff_{q^n})}\sum_{i=0}^{n-1}\Fr_{X,n}^i(\rho)(\O_{X_n}(x')) \cdot \\
& & \qquad\cdot \sum_{j=0}^{n-1}\ov{\Fr_{X,n}^j(\sigma)(\O_{X_n}(x'))}\\
& = & \frac{|X(\ff_{q^n})|}{n^2}\sum_{i,j=0}^{n-1}\<\Fr_{X,n}^i(\rho),\Fr_{X,n}^j(\sigma)\>\\
& = & \left\{\begin{array}{ll}
	\ds\frac{|X(\ff_{q^n})|}n & \mathrm{ if }\, \trho=\tsigma\\
	0 & \mathrm{ otherwise}
\end{array}\right.
\end{eqnarray*}
Putting all together we get the announced formula.
\end{proof}

\begin{lemma}\label{L:normalization of cusp eigenforms} Let $f$ be a (non zero) cusp eigenform of rank $n$. Suppose we can write $f=\sum_{d\in\ZZ}f_{dn}$ where each $f_{dn}$ is supported on vector bundles of degree $nd$. Then $f(\O^{\oplus n})\neq 0$ and therefore we can renormalize $f$ such that $f(\O^{\oplus n})=1$.
\end{lemma}
\begin{proof}
Let us first make a small remark. By the eigenform property it follows that $f_{dn} = \zeta^d f_0\ot_\O \O_X(dx_0)$, where $\zeta$ is the value by which $\He_{\O_{x_0}^{\oplus n}}$ acts on $f$, and therefore all the terms $f_{dn}$ are completely determined by $f_0$ and $\zeta$. In particular, if $f_n=0$ then $f=0$.

Let $\tau$ be a torsion sheaf of degree $n$ such that $f(\V(\tau))\neq 0$. Such a sheaf must exist since all the semistables of slope 1 and rank $n$ are of the form $\V(\tau)$ for some $\tau$ and $f_n$ is supported on the semistables and is non zero.

We have the following equalities:
\begin{eqnarray*}
\overline{\chi(\tau)}(f,\O^{\oplus n}) & = & (\He_\tau^{\vee}(f),\O^{\oplus n}) \\
& = & (f,\He_\tau(\O^{\oplus n}))\\
& = & (f,(\He_\tau(\O^{\oplus n}))^\ss)\\
& \overset{Cor.~\ref{C:Hecke tau applied to O^n}}{=} & P_{\tau,\O^{\oplus n}}^{\V(\tau)}(f,\V(\tau))\\
& \neq &  0
\end{eqnarray*}
where we denoted by $\chi(\tau)$ the constant by which $\He_\tau$ acts on $f$. Therefore $f(\O^{\oplus n})\neq 0$ and this is what we wanted to prove.
\end{proof}

\begin{rem}
We have used here the dual Hecke operators $\He_\tau^\vee$ which are just adjoints with respect to the Green product of the Hecke operators. For more details see \cite{Kap} Section 2.6.
\end{rem}

\begin{lemma}\label{L:norm-surjective} If $Y$ is an elliptic curve over $\ff_q$ then the norm map $$\Norm_n: Y(\ff_{q^n})\to Y(\ff_q)$$ is surjective for every $n$.
\end{lemma}

\begin{proof} It is enough to prove that the map $\Norm_n:\ov Y\to \ov Y$ is surjective (on the closed points). Indeed, if we prove this it follows that there is some $N$ such that $\Norm_n(Y(\ff_{q^N})\supseteq Y(\ff_q)$. For a point $y\in Y(\ff_{q^N})$ we have that $\Norm_n(y) = y\oplus \Frob_Y(y)\oplus\dots\oplus\Frob_Y^{n-1}(y)$ and therefore if $\Norm_n(y)\in Y(\ff_q)$ it follows that $\Frob_Y\circ\Norm_n(y)=\Norm_n(y)$ or in other words $y=\Frob_Y^n(y)$. This is equivalent to $y$ being in $Y(\ff_{q^n})$. 

Now the map $\Norm_n:\ov Y\to \ov Y$ is a homomorphism of elliptic curves. It follows that its image is either a point or the entire curve. Obviously it cannot be a point because the kernel of $\Norm_n$ is contained in $Y(\ff_{q^n})$ and therefore it is finite. Hence the map $\Norm_n:\ov Y\to\ov Y$ is surjective.
\end{proof}
\begin{rem} The same result with the same proof holds for any abelian variety over $\ff_q$.
\end{rem}

\begin{lemma}\label{L:scalar product of products of elements}
Let $a_i\in \HH_X^{(\nu_i)}, i=1,\dots, r$ and $b_j\in\HH_X^{(\nu'_j)},j=1,\dots,s$ be homogeneous elements, where $\nu_i,\nu'_j\in\Q\cup\{\infty\}$ are such that $\nu_1 < \ldots < \nu_r$ and $\nu'_1 < \ldots < \nu'_s$. Then we have 
\[
(a_1\dots a_r,b_1...b_s) = \left\{
\begin{array}{ll}
0, & \mathrm{if }\, (\nu_1,\dots,\nu_r)\neq (\nu'_1,\dots,\nu'_s)\\
(a_1,b_1)(a_2,b_2)\dots(a_r,b_r), & \mathrm{ otherwise }
\end{array}\right.
\]
\end{lemma}
\begin{proof}
Let us denote temporarily by $\alpha$ the value of the scalar product we want to compute.
By the $\SL_2(\ZZ)$ symmetry we can suppose that $\nu_r=\infty$. It is clear that if $\nu'_s<\infty$ then $\alpha=0$. Hence we can suppose that $\nu'_s = \infty$. In this case it is easy to see using the ordering of the slopes and the Hopf property of the Green form that:
\[
 \alpha = (a_r,b_s)(a_1\cdots a_{r-1},b_1\cdots b_{s-1})
\]
By induction we obtain that $\alpha$ is 0 if $(\nu_1,\dots,\nu_r)\neq (\nu'_1,\dots,\nu'_s)$ and $\alpha$ equals $(a_1,b_1)\cdots (a_r,b_r)$ otherwise.
\end{proof}

\begin{lemma}\label{L:restr of rep over Gbar are semisimple}
Let $G:=\pi_1(X)$ where $X/\ff_q$ is an elliptic curve and let $\ov G:=\pi_1(\ov X)$. Let also $V$ be an irreducible continuous representation of $G$ over $\ov\Q_l$. Then the restriction of $V$ to $\ov G$ is a direct sum of characters.
\end{lemma}
\begin{proof}
Since $X$ is an elliptic curve we have that $\ov G$ is an abelian group. By Schur's Lemma it follows that all the irreducible representations of $\ov G$ appearing in $V$ are of dimension 1. Denote by $V^\circ\subseteq V$ the socle of $V$ as a $\ov G$ module. On $V$ we have also an action of the Galois group $\Gal(\fqb/\ff_q)$ and this action permutes the irreducible representations of $\ov G$ appearing in $V$. It follows that $\Gal(\fqb/\ff_q)$ must leave $V^\circ$ stable by the definition of the socle (the sum of all the simple submodules). Therefore $V^\circ$ is a $G$-submodule of $V$. From the irreducibility of $V$ as a $G$-module we conclude that $V^\circ = V$ and therefore that $V$, as a $\ov G$ module, is a direct sum of 1-dimensional representations.
\end{proof}

\begin{center}\textsc{Acknoledgements}
\end{center}
This work is part of my PhD thesis. I would like to thank Olivier Schiffmann, my supervisor, for suggesting this interesting problem, for generously sharing his ideas with me and for all his patient help along the way. I also benefited from an exchange of emails between G. Laumon and my advisor concerning the Langlands correspondence for elliptic curves. I would also like to thank Michael Gr\"ochenig and Alexandre Bouayad for interesting discussions.

\end{document}